\documentclass[10pt]{article}
\usepackage[latin1]{inputenc}
\usepackage{amsmath,amsthm,amssymb,amsfonts,amsxtra}
\usepackage{graphicx}

\newtheorem{theorem}{Theorem}[section]
\newtheorem*{theorem*}{Theorem}
\newtheorem{lem}{Lemma}[section]
\newtheorem{corollary}{Corollary}[section]

\newtheorem{definition}[theorem]{Definition}

\newcommand{\setR}{\mathbb{R}}
\newcommand{\setN}{\mathbb{N}}
\newcommand{\setC}{\mathbb{C}}
\newcommand{\setZ}{\mathbb{Z}}
\newcommand{\XXint}[3]{{\setbox0=\hbox{$#1{#2#3}{\int}$}
      \vcenter{\hbox{$#2#3$}}\kern-.5\wd0}}

\newcommand{\settmp}[2]{#1\{{#2}#1\}}
\newcommand{\set}[1]{\settmp{}{#1}}
\newcommand{\bigset}[1]{\settmp{\big}{#1}}

\newcommand{\normtmp}[2]{#1\lVert{#2}#1\rVert}
\newcommand{\norm}[1]{\normtmp{}{#1}}

\newcommand{\Bignorm}[1]{\normtmp{\Big}{#1}}

\newcommand{\abstmp}[2]{#1\lvert{#2}#1\rvert}
\newcommand{\abs}[1]{\abstmp{}{#1}}

\newcommand{\Bigabs}[1]{\abstmp{\Big}{#1}}

\newcommand{\Div}{\operatorname{div}} 
\newcommand{\curl}{\operatorname{curl}}

\newcommand{\metai}{\overline{m}_\eta}
\newcommand{\metaii}{\overline{\overline{m}}_\eta}
\newcommand{\mi}{\overline{m}}
\newcommand{\mii}{\overline{\overline{m}}}
\newcommand{\ui}{\overline{u}}
\newcommand{\uii}{\overline{\overline{u}}}
\newcommand{\wi}{\overline{m}_1}
\newcommand{\wii}{\overline{m}_2}
\newcommand{\wiii}{\overline{m}_3}
\newcommand{\uwi}{\overline{w}}
\newcommand{\uwii}{\overline{\overline{w}}}
\newcommand{\vi}{\overline{u}_1}
\newcommand{\vii}{\overline{u}_2}
\newcommand{\viii}{\overline{u}_3}

\newcommand{\dt}{\mathbf{T}} 
\newcommand{\Poincare}{Poincar\'e}
\newcommand{\diag}{\operatorname{diag}}

\title{\sc Periodic solutions for the Landau-Lifshitz-Gilbert equation\
 }
\author{\sc  Alexander Huber}

\date{}

\begin{document}

\maketitle

\begin{abstract}
Ferromagnetic materials tend to develop very complex magnetization patterns whose time evolution is modeled by the so-called Landau-Lifshitz-Gilbert equation (LLG). In this paper, we construct time-periodic solutions for LLG in the regime of soft and small ferromagnetic particles which satisfy a certain shape condition. Roughly speaking, it is assumed that the length of the particle is greater than its hight and its width. The approach is based on a perturbation argument and the spectral analysis of the corresponding linearized problem as well as the theory of sectorial operators.
\end{abstract}

\section{Introduction}
Ferromagnetic materials show a large variety of magnetic microstructures, which can be made visible with the help of the Magneto-optical Kerr effect microscopy (we refer to the book by Hubert and Sch\"afer \nolinebreak \cite{HS} for a description of this microscopy method and plenty of beautiful magneto-optical images). Roughly speaking, one observes regions where the magnetization is almost constant 
(named magnetic domains) and transition layers separating them (known as domain walls).

The theory of micromagnetism, as presented in the book by 
Brown \cite{brown} (see also \cite{HS} and the book by Aharoni \cite{aharoni2})
explains the multiple magnetic phenomena (in the static case) by (local) minimization 
of a certain energy functional, the micromagnetic energy $E$. The quantity to be predicted is the magnetization vector $m:\Omega \to J_s S^2 \subset \setR^3$ of a ferromagnetic sample $\Omega \subset \setR^3$, where
the pointwise constraint $\abs{m(x)}=J_s$ (also known as ``saturation constraint'') reflects the experimentally confirmed fact that the magnitude of the magnetization is constant (the positive scalar $J_s$ is called saturation magnetization). In suitable units, the micromagnetic energy $E(m)$ associated with the magnetization vector $m$ is given by
\begin{align*}
 E(m) = d^2 \int_\Omega \abs{\nabla m}^2 \, dx+ Q \int_\Omega \varphi(m) \, dx
  + \int_{\setR^3} \abs{H[m]}^2 \, dx
  - 2 \int_\Omega h_{\text{ext}} \cdot m \, dx
\end{align*}
and we can assume w.l.o.g. that $J_s=1$, i.e., $m$ has values in the unit sphere $S^2$.
On the right hand side, the four terms are called exchange energy, anisotropy energy, stray field energy, and external field energy, respectively. The exchange energy explains the tendency towards parallel alignment, where the positive material constant \nolinebreak $d$ is called exchange length or Bloch line width. Preferred directions due to crystalline anisotropies -- so-called easy axis -- are modeled by the material function $\varphi:S^2 \to [0,\infty)$ and the positive quality factor $Q$. 
Furthermore, the magnetization $m$ induces a magnetic field $H[m]$ -- the so-called stray field -- that solves the Maxwell equations of magnetostatics
\begin{align}
\label{eq:Maxwell}
 \begin{array}{rl}
    \Div \big( H[m] + \chi_{\Omega} \, m \big) &= 0
    \\[-0,2cm]
    \curl H[m] &= 0
 \end{array} \quad \text{in } \setR^3
\end{align}
and whose energy contribution is given by the stray field energy.
Finally, the external field energy captures the interactions of $m$ with an external magnetic field $h_\text{ext}$.

In view of their different behavior, the four energies can
never simultaneously achieve their minimal values. For example, the exchange energy prefers constant magnetizations, whereas the stray field energy favors divergence free magnetizations. The resulting 
competition between the four energies explains (some of) the observed microstructures in 
ferromagnetic materials.

Most of the (qualitative) mathematical theory so far has focused on the static case (an extensive
list of references can be found in the survey article by DeSimone, Kohn, M{\"u}ller, and Otto \nolinebreak \cite{dkmo}).
Here instead, the corresponding evolution equation with regard to time-periodic solutions is studied. In the physics community, this equation is known as Landau-Lifshitz-Gilbert equation (LLG) and given by
\begin{align*}
    m_t = \alpha \, m \times H_{\text{eff}} - m \times \big( m \times H_{\text{eff}} 
    \big) \quad\text{in } \setR \times \Omega
\end{align*}
(\textquotedblleft$\times$\textquotedblright{} denotes the usual vector product in $\setR^3$) together with homogeneous Neumann boundary conditions
\begin{align*}
 \frac{\partial m}{\partial \nu} = 0 \quad \text{on } \setR \times \partial \Omega
\end{align*}
and the saturation constraint $\abs{m}=1$. Here $2 H_{\text{eff}} = -\nabla_{L^2} E(m) $ is the effective magnetic field and $\alpha \in \setR$ is a parameter. The so-called ``gyromagnetic'' term $\alpha \, m \times H_{\text{eff}}$ describes 
a precession around $H_{\text{eff}}$, whereas the ``damping'' term $- m  \times \big( m \times H_{\text{eff}}\big) $ tries to 
align $m$ with $H_{\text{eff}}$. Mathematically speaking, the Landau-Lifshitz-Gilbert equation is a hybrid heat and Schr{\"o}dinger flow for the free energy $E$.
The paper at hand addresses the following question:
\\[1ex]
\centerline{\itshape{Do there exist time-periodic solutions for LLG when the external magnetic field is periodic in time?}}
\\[1ex]
Due to the complexity of LLG, one can not expect to obtain a universally valid answer to that question. Here we restrict ourself to the regime of soft and small ferromagnetic particles. \textquotedblleft{}Soft\textquotedblright{} refers to the case when $Q=0$, and \textquotedblleft{}small\textquotedblright{} means that $\abs{\Omega} \ll 1$, where $\abs{\Omega}$ denotes the three-dimensional Lebesgue measure of $\Omega$. In this case, the effective magnetic field $H_\text{eff}$ is given by
\begin{align*}
 H_{\text{eff}} = d^2 \Delta m + H[m] + h_{\text{ext}}
\end{align*}
with stray field $H[m]$ and external magnetic field $h_\text{ext}$. If we neglect the precession term (i.e. $\alpha =0$), LLG is simply the harmonic map heat flow into the sphere with an additional lower order (but non-local) term. We therefore expect that in general solutions of LLG tend to develop singularities in finite time (see for example the recent book by Guo and Ding \cite{GuoDing} and references therein) which, of course, makes it a difficult task to construct time-periodic solutions. Even in the case $h_\text{ext}=0$, the most natural time-periodic solutions for LLG -- namely minimizers of the micromagnetic energy functional -- are not regular in general (see Hardt, Kinderlehrer \cite{hardtkinderlehrer} and Carbou \cite{carbou}).

The situation is more accommodating in the case of small ferromagnetic particles. Indeed, a scaling argument shows that the exchange energy determines the behavior of the magnetization to a large extent. Thus, it is to be expected that for very small samples, minimizers of the micromagnetic energy functional are nearly constant (but still not trivial). In \cite{DeSimone} a reduced model for small bodies is derived by means of $\Gamma$-convergence, and it turns out that the celebrated model of Stoner and Wohlfarth \cite{Stoner} is recovered. As expected, the limiting energy functional is defined on a finite-dimensional space, which makes it easy to study its local and global minimizers. Moreover, as shown in \cite{DeSimone}, the limiting energy functional captures the asymptotic behavior of local and global minimizers of the full energy functional as the size of the particle tends to zero. 
For small particles, it is therefore possible to show that minimizers of the micromagnetic energy functional are regular, provided $\partial \Omega$ is smooth enough (see \cite{alex_regularity}, \cite{alex_diss}).

In this paper, regular minimizers of the micromagnetic energy functional form our starting point for the construction of time-periodic solutions. We make use of a perturbation argument based on the spectral properties of the corresponding linearized problem and the implicit function theorem. Besides the fact that we work with small particles, we also have to impose additional assumptions on the shape of the ferromagnetic sample in terms of the eigenvalues of the demagnetizing tensor $\dt$ (see Section \ref{sec:spectral analysis - second step}). The resulting magnetic shape anisotropy is used to rule out nontrivial zeros for the linearization. As a consequence, we can show the existence of $T$-periodic solutions for LLG in the regime of soft and small ferromagnetic particles satisfying a certain shape condition, where at the same time it is assumed that the amplitude of the $T$-periodic external magnetic field is sufficiently small (see Section \nolinebreak \ref{sec:main result - small particles}).
We note that a related existence result for N{\'e}el walls in thin ferromagnetic films is proved in \cite{alex_walls}.

The paper is organized as follows:
In Section \ref{sec:energy functional for small particles} we introduce the rescaled LLG for small particles and show the existence of regular solutions close to a regular minimizer in Section \ref{sec:existence of solutions - small particles} with the help of analytic semigroups and optimal regularity theory. We explain our perturbation argument in Section \nolinebreak \ref{sec:continuation} and derive the necessary spectral properties of the linearization in Sections \ref{sec:spectral analysis - first step} and \ref{sec:spectral analysis - second step}. This is used in Section \ref{sec:main result - small particles} to prove the existence of time-periodic solutions.

\section{Energy functional and LLG for small particles}
\label{sec:energy functional for small particles}
\paragraph{Scaling of the energy functional and regularity.}
We consider a small and soft ferromagnetic particle modeled by $\Omega_\eta = \eta \, \Omega \subset \setR^3$. Here, $\Omega \subset \setR^3$ is a bounded domain with $\abs{\Omega}=1$, and $\eta>0$ is a (small) parameter representing the size of the particle.
In order to have a magnetization vector defined on a fixed domain, we rescale space by
\begin{align*}
 x \mapsto \frac{1}{\eta} x
\end{align*}
and obtain -- after a further renormalization of the energy by $\eta$ -- the rescaled energy functional (here stated with $h_\text{ext}=0$)
\begin{align*}
 E_{\text{res}}^\eta (m) =  d^2 \int_\Omega \abs{\nabla m}^2 \, dx+ \eta^2 \int_{\setR^3} \abs{H[m]}^2 \, dx\, .
\end{align*}
The rescaled magnetization $m:\Omega \to S^2$ is now defined on a fixed domain independent of $\eta$, and the corresponding stray field $H[m]$ is a solution of the static Maxwell equations \eqref{eq:Maxwell}.
In the sequel we vary the (small) parameter $\eta$ and set $d=1$ for convenience.

With standard methods from the calculus of variations, it can easily be seen that $E_{\text{res}}^\eta$ admits a minimizer $m_\eta$ in the set of admissible functions defined by
\begin{align*}
 H^1(\Omega,S^2) = \bigset{m \in H^1(\Omega,\setR^3) \, \big| \,  \abs{m}=1 \text{ almost everywhere}} \, .
\end{align*}
We remark that the stray field $H[u]$ for functions $u \in L^2(\Omega,\setR^3)$ is uniquely determined by the requirement $H[u] \in L^2(\setR^3,\setR^3)$. Moreover, we have the following basic lemma for the stray field operator \nolinebreak $H$:
\begin{lem}
\label{lem:stray field}
 The stray field operator $H$ defines a linear and bounded mapping from
 $L^2(\Omega,\setR^3)$ to $L^2(\setR^3,\setR^3)$ and satisfies the identity
\begin{align*}
 \int_{\setR^3} H[u]\cdot H[v] \, dx= - \int_\Omega H[u] \cdot v\, dx
\end{align*}
for every $u,v \in L^2(\Omega,\setR^3)$. In particular, the mapping $u \mapsto H[u]_{|\Omega}$ is symmetric on $L^2(\Omega,\setR^3)$.
\end{lem}
\begin{proof}
 The proof is standard (see for example \cite{dkmo}). We remark that $H[u]$ is actually a gradient, and therefore, we obtain the stated identity 
with the help of the weak formulation for $H[v]$.
\end{proof}
Furthermore, we know that minimizers of $E_{\text{res}}^\eta$ become ``regular'' if $\eta>0$ is sufficiently small. More precisely, the following theorem holds (consult \cite{alex_regularity} or \cite{alex_diss} for a proof):
\begin{theorem}
  \label{thm:regularity}
  Let $\Omega \subset \setR^3$ be a bounded $C^{2,1}$-domain. There exist positive constants $\eta_0 = \eta_0(\Omega)$ and $C_0=C_0(\Omega)$ with the following property: If $m_\eta$ is a minimizer of $E^\eta$ on $H^1(\Omega,S^2)$ with parameter $0<\eta\le\eta_0$, then 
  \begin{align}
   \label{eq:regularity}
   m_\eta \in H^2_N(\Omega,\setR^3) \cap C^{1,\gamma}(\overline{\Omega},\setR^3) \qquad \text{and} \qquad \norm{\nabla m_\eta}_{L^\infty} \le C_0 \eta
  \end{align}
  for every $\gamma \in (0,1)$, where ``$N$'' stands for homogeneous Neumann boundary conditions $\frac{\partial m_\eta}{\partial \nu} =0$ ($\nu$ is the outer normal of $\partial \Omega$). Moreover, $m_\eta$ satisfies the Euler-Lagrange equation
 \begin{align}
 \label{eq:Euler-Lagrange}
  \Delta m_\eta +  \abs{\nabla m_\eta}^2 m_\eta - \eta^2 m_\eta \cdot H[m_\eta] \, m_\eta + \eta^2 H[m_\eta] = 0 \, .
\end{align}
\end{theorem}
We call minimizers $m_\eta$ of $E_\text{res}^\eta$ which satisfy property \eqref{eq:regularity} \textquotedblleft{}regular minimizers\textquotedblright{}.
In the sequel we frequently make use of the vector identities
\begin{gather}
\label{eq:vector identities}
(a \times b) \cdot c = - (c \times b) \cdot a \quad \text{and} \quad
a \times (b \times c) = (a \cdot c) b - (a\cdot b) c
\end{gather} 
for $a,b,c \in \setR^3$.
In particular, we have that
\begin{align}
 \label{eq:identity for Laplacian}
 \Delta m + \abs{\nabla m}^2 m = - m \times ( m \times \Delta m)
\end{align}
for all $m \in H^2(\Omega,\setR^3)$ with $\abs{m}=1$, and we can rewrite the Euler-Lagrange equation \eqref{eq:Euler-Lagrange} as follows:
\begin{align*}
 m_\eta \times \left( m_\eta \times \left( \Delta m_\eta + \eta^2 H[m_\eta]\right) \right) = 0 \,.
\end{align*}
From here we see that
\begin{align}
 \label{eq:m is parallel to H effective}
 m_\eta \parallel  \Delta m_\eta + \eta^2 H[m_\eta]
\end{align}
for all regular minimizers $m_\eta$ of $E_{\text{res}}^\eta$, where \textquotedblleft$\parallel$\textquotedblright{} stands for \textquotedblleft{}parallel\textquotedblright.
\paragraph{Rescaled LLG.} Now, we consider LLG for our soft ferromagnetic particle $\Omega_\eta$. If we rescale space and time by
\begin{align*}
 x \mapsto \frac{1}{\eta} x \qquad \text{and} \qquad t \mapsto \frac{1}{\eta^2} t \, ,
\end{align*}
respectively, we obtain the rescaled LLG
\begin{align}
 \label{eq:rescaled LLG}
 m_t = \alpha \, m \times H_\text{eff}^\text{res} - m \times ( m \times H_\text{eff}^\text{res}) \quad \text{in } \setR \times \Omega
\end{align}
together with homogeneous Neumann boundary conditions and saturation constraint $\abs{m}=1$. Again, the rescaled magnetization $m:\setR \times \Omega \to S^2$ is now defined on a fixed domain independent of the parameter $\eta$, and the rescaled effective magnetic field $H_\text{eff}^\text{res}$ is given by
\begin{align*}
 H_\text{eff}^\text{res} = \Delta m + \eta^2 H[m] + \eta^2 h_\text{ext} \,.
\end{align*}
With the help of \eqref{eq:m is parallel to H effective}, we see that regular minimizers of $E_{\text{res}}^\eta$ are stationary solutions for the rescaled LLG if $h_\text{ext}=0$.

In the following we assume throughout that $\Omega$ is a bounded $C^{2,1}$-domain and that $\eta$ is sufficiently small. In order to prove the existence of $T$-periodic solutions for the rescaled LLG in the case of $T$-periodic external magnetic fields with small amplitude, we perturb a regular minimizer of $E_{\text{res}}^\eta$ and introduce a new parameter $\lambda$. To be more precise, we replace the external magnetic field $h_\text{ext}$ by $\lambda  h$, where $h:\setR \times \Omega \to \setR^3$ is $T$-periodic in time and $\lambda \in \setR$.

\section[Existence of solutions and dependence on \dots]{Existence of solutions and dependence on the data}
\sectionmark{Existence of solutions and dependence on \dots}
\label{sec:existence of solutions - small particles}
In this section we show the existence of solutions for the rescaled LLG close to a regular minimizer
$m_\eta$ of $E_{\text{res}}^\eta$ by means of the implicit function theorem and optimal regularity theory for the linearized problem. Moreover, we prove the differentiability of these solutions with respect to $\lambda$ and the initial value. Our proof employs the theory of sectorial operators and analytic semigroups (see the book by Lunardi \nolinebreak \cite{lunardi} for a self-contained presentation). First of all, we recall the central definition of this theory:
\begin{definition}
\label{def:sectorial}
 Let $X$ be a Banach space and $A:D(A) \subset X \to X$ be a linear operator. We say that $A$ is sectorial if there are constants $\omega \in \setR$, $\theta \in (\pi/2,\pi)$, and $M>0$ such that the resolvent set $\rho(A)$ contains the sector $S_{\theta,\omega} = \set{ \lambda \in \setC \, | \, \lambda \not= \omega, \, \abs{\arg(\lambda - \omega)} < \theta}$,
and the resolvent estimate
\begin{align*}
 \norm{\text{R}(\lambda,A)} \le \frac{M}{\abs{\lambda - \omega}}
\end{align*}
is satisfied for all $\lambda \in  S_{\theta,\omega}$. Moreover, we denote by $(e^{tA})_{t \ge 0}$ the analytic semigroup generated by the sectorial operator $A$.
\end{definition}
\noindent We remark that sectorial operators $A:D(A) \subset X \to X$ are closed, and the domain of definition $D(A)$ becomes a Banach space when equipped with the graph norm $\norm{\cdot}_{D(A)}$ of $A$ defined as usual by
 $\norm{x}_{D(A)} = \norm{x}_X + \norm{Ax}_X$
for $x \in D(A)$. 

In the sequel we work with an equivalent version for the rescaled LLG given by
\begin{equation}
 \label{eq:new version for rescaled LLG}
m_t = \Delta m + \alpha \, m \times \Delta m + \abs{\nabla m}^2 m + \alpha \, \eta^2 \, m \times ( H[m] + \lambda h) 
- \eta^2 \, m \times \big( m \times (H[m] + \lambda h) \big)
\end{equation}
together with homogeneous Neumann boundary conditions and saturation constraint $\abs{m}=1$.
This is indeed equivalent to \eqref{eq:rescaled LLG} thanks to \eqref{eq:identity for Laplacian}.
To be able to work in linear
 spaces, we first forget about the pointwise constraint $\abs{m}=1$ and show a posteriori that this
 is true if it is satisfied by the initial value. For $\lambda = 0$ the linearization $\mathcal{L}^\eta$ of the right hand side of \eqref{eq:new version for rescaled LLG} with respect to $m$ at $m_\eta$ is given by
\begin{align}
\label{eq:linearization}
\begin{split}
 \mathcal{L}^\eta u =& \Delta u + \alpha \, m_\eta \times \Delta u 
+ 2 \, \nabla u : \nabla m_\eta \, m_\eta + \abs{\nabla m_\eta}^2 u  + \alpha \, u \times \Delta m_\eta + \alpha \, \eta^2 \, m_\eta \times H[u] 
\\
&+ \alpha \, \eta^2 \, u \times H[m_\eta] - \eta^2 \, m_\eta \times \big( m_\eta \times H[u] \big) - \eta^2 \, m_\eta \times \big( u \times H[m_\eta] \big) - \eta^2 \, u \times \big (m_\eta \times H[m_\eta] \big)
\end{split}
\end{align}
for $u \in H^2_N(\Omega,\setR^3)$. We remark that the linear operator $\mathcal{L}^\eta$ is well-defined as a mapping from $H^2_N(\Omega,\setR^3)$ to $L^2(\Omega,\setR^3)$ thanks to the embedding $H^2(\Omega,\setR^3) \hookrightarrow L^\infty(\Omega,\setR^3)$. Furthermore, we can show that $\mathcal{L}^\eta$ is sectorial:
\begin{lem}
 \label{lem:SmallSectorial}
 The linear operator $\mathcal{L}^\eta : H^2_N(\Omega,\setR^3) \subset L^2(\Omega,\setR^3) \to
 L^2(\Omega,\setR^3)$ is sectorial, and the graph norm of $\mathcal{L}^\eta$ is equivalent to the $H^2$-norm.
\end{lem}
\begin{proof}
 We use the decomposition $\mathcal{L}^\eta = \mathcal{L}^\eta_1 + \mathcal{L}^\eta_2$, where
 \begin{align*}
 \mathcal{L}^\eta_1 : H^2_N(\Omega,\setR^3) \subset L^2(\Omega,\setR^3) \to
 L^2(\Omega,\setR^3): 
 u \mapsto \Delta u + \alpha \, m_\eta \times \Delta u \, ,
 \end{align*}
 and show that
 $\mathcal{L}^\eta_1$ is sectorial. For this we use the definitions and notations from \cite{Amann2}
 and write
\begin{align*}
 \mathcal{L}^\eta_1 u = A \Delta u = 
 \left(
\begin{array}{ccc}
 1 & -\alpha m^\eta_3 & \hphantom{-}\alpha m^\eta_2
 \\
 \hphantom{-}\alpha m^\eta_3 & 1 & -\alpha m^\eta_1
 \\
 -\alpha m^\eta_2 & \hphantom{-}\alpha m^\eta_1 & 1
\end{array}
 \right)
\Delta u = \sum_{j,k=1}^3 a_{jk} \partial_j\partial_k u
\, ,
\end{align*}
where $a_{jk} = \delta_{jk} A \in C \big(\overline{\Omega}, \setR^{3 \times 3} \big)$. For $x \in \overline{\Omega}$, $\xi \in \setR^3\setminus\set{0}$, and $\zeta \in \setC^3 \setminus \set{0}$ we find
\begin{align*}
 \sum_{j,k,r,s} a_{jk}^{rs}(x) \xi^j \xi^k \zeta_r \overline{\zeta}_s = \sum_{j,r,s} A^{rs}(x)
 \big(\xi^j\big)^2 \zeta_r \overline{\zeta}_s = \abs{\xi}^2 \sum_{r,s} A^{rs}(x) \zeta_r
 \overline{\zeta}_s \,.
\end{align*}
With the decomposition $\zeta = \zeta^1 + \dot{\imath}\zeta^2$, where $\zeta^1$, $\zeta^2 \in \setR^3$, it is easy to see that
\begin{align*}
 \text{Re} \sum_{j,k,r,s} a_{jk}^{rs}(x) \xi^j \xi^k \zeta_r \overline{\zeta}_s = \abs{\xi}^2
 \abs{\zeta}^2 >0 \,.
\end{align*}
This shows that the symbol of $\mathcal{L}^\eta_1$ is strongly uniformly elliptic. If we set $\mathcal{A} = \mathcal{L}^\eta_1$ and $\mathcal{B}= \partial_\nu$ ($\nu$ is the unit outer normal of $\partial \Omega$), then we obtain with the help of  
 \cite[Theorem 4.2]{Amann2} that $(\mathcal{A},\mathcal{B})$ is a normally elliptic boundary value problem of second order. Now, \cite[Theorem 2.4]{Amann2} implies that the $L^2$-realization of $(\mathcal{A},\mathcal{B})$, i.e., the operator $\mathcal{L}^\eta_1$,
is sectorial. Moreover, the open mapping theorem implies that the graph norm of $\mathcal{L}^\eta_1$ is equivalent to the $H^2$-norm.

Next, we show that $\mathcal{L}^\eta_2$ is a lower order perturbation of $\mathcal{L}^\eta_1$. For this we first recall that $\nabla m_\eta$ belongs to $L^\infty\big(\Omega,\setR^{3\times3}\big)$, and thanks to \cite[Lemma 2.3]{carboufabrie}, we know that the
restricted stray field operator
\begin{align*}
  H: H^2(\Omega,\setR^3) \to H^2(\Omega,\setR^3) : m \mapsto H[m]_{|\Omega}
\end{align*}
is linear and bounded. In particular, we find $H[m_\eta] \in L^\infty(\Omega,\setR^3)$, and the Euler-Lagrange equation \eqref{eq:Euler-Lagrange} implies
$\Delta m_\eta \in L^\infty(\Omega,\setR^3)$. A combination of these facts together with the Sobolev embedding theorem shows that
$\mathcal{L}^\eta_2 : H^2_N(\Omega,\setR^3) \to L^2(\Omega,\setR^3)$
is a compact operator. Now, the perturbation result \cite[Proposition 2.4.3]{lunardi} implies that $\mathcal{L}^\eta$ is sectorial, and similarly as above, we see that the graph norm of $\mathcal{L}^\eta$ is equivalent to the $H^2$-norm. The lemma is proved.
\end{proof}
In order to state the announced optimal regularity result, we have to introduce weighted function spaces of H{\"o}lder continuous functions (we use the notation from \cite{lunardi}). Let therefore $0<\beta<1$ be given and $X$ be a Banach space. The space $C^{0,\beta}_\beta(]0,T],X)$ is the set of all bounded functions $f:\,]0,T] \to X$ such that
\begin{align*}
 [f]_{C^{0,\beta}_\beta(]0,T],X)} = \sup_{0<\epsilon<T} \epsilon^\beta [f]_{C^{0,\beta}([\epsilon,T],X)}
 < \infty \,.
\end{align*}
This forms a Banach space with norm defined by
\begin{align*}
\norm{f}_{C^{0,\beta}_\beta(]0,T],X)} = \norm{f}_{C(]0,T],X)} + [f]_{C^{0,\beta}_\beta(]0,T],X)} \,.
\end{align*}
Moreover, we write $C^{1,\beta}_\beta (]0,T],X)$ for the set of all bounded and differentiable functions $f :\,]0,T] \to X$ with derivative $f' \in C^{0,\beta}_\beta(]0,T],X)$. Again, this is a Banach space with norm defined by
\begin{align*}
\norm{f}_{C^{1,\beta}_\beta (]0,T],X)} = \norm{f}_{C(]0,T],X)} + \norm{f'}_{C^{0,\beta}_\beta(]0,T],X)} \,.
\end{align*}
 We can now state the required optimal regularity result which is a combination of \cite[Corollary 4.3.6]{lunardi} (ii) and (iii):
\begin{lem}
 \label{lem:optimal regularity}
 Suppose $X$ is a Banach space and $A:D(A) \subset X \to X$ is a densely defined sectorial operator.
Let $0<\beta<1$, $0<T<\infty$, $u_0 \in D(A)$, and
$f \in C([0,T],X) \cap C^{0,\beta}_\beta(]0,T],X)$ be given. Furthermore, let
 $u$ be the mild solution of
 \begin{align*}
  u_t = Au + f(t), \quad u(0) = u_0 \,,
 \end{align*}
 by definition
 \begin{align*}
  u(t) = e^{tA} u_0 + \int_0^t e^{(t-s)A} f(s) \, ds \, , \qquad 0 \le t \le 
  T \,.
 \end{align*}
 Then $u$ is a strict solution, this means $u \in C^1([0,T],X) \cap C\big([0,T],D(A)\big)$ and the
 above equation holds pointwise. Moreover, we have the optimal regularity result $u_t$, $A u \in
 C([0,T],X) \cap C^{0,\beta}_\beta(]0,T],X)$.
\end{lem}
Next, we define the Banach spaces
\begin{align*}
 X = C^1([0,T],L^2) \cap C([0,T],H^2_N) \cap C^{0,\beta}_\beta(]0,T],H^2_N) \cap
 C^{1,\beta}_{\beta}(]0,T],L^2)
\end{align*}
and
\begin{align*}
 Y = C([0,T],L^2) \cap C^{0,\beta}_\beta(]0,T],L^2)
\end{align*}
with respective norms
\begin{align*}
 \norm{\cdot}_X = \norm{\cdot}_{C^1([0,T],L^2)} + \norm{\cdot}_{C([0,T],H^2_N)} 
 + \norm{\cdot}_{C^{0,\beta}_\beta(]0,T],H^2_N)} + \norm{\cdot}_{C^{1,\beta}_{\beta}(]0,T],L^2)}
\end{align*}
and
\begin{align*}
 \norm{\cdot}_Y = \norm{\cdot}_{C([0,T],L^2)} + \norm{\cdot}_{C^{0,\beta}_\beta(]0,T],L^2)} \,.
\end{align*}
We come to our existence result:
\begin{lem}
 \label{lem:existence}
 Let $0<\beta<1$ and $h \in C^{0,\beta}\big(\setR,L^2(\Omega,\setR^3) \big)$ be given.
 Furthermore, let $m_\eta$ be a regular minimizer of $E_{\text{res}}^\eta$ with parameter $\eta>0$. Then
 there exists an open neighborhood $U_\eta$ of $m_\eta$ in $H^2_N(\Omega,\setR^3)$ and an open neighborhood
 $V_\eta$ of $0$ in $\setR$ such that
 \begin{align*}
 \tag*{$(LLG)_\eta$}
  m_t = \Delta m + \alpha \, m \times \Delta m + \abs{\nabla m}^2 m + \alpha \, \eta^2 \, m \times (
 H[m] + \lambda h) - \eta^2 \, m \times \big( m \times (H[m] + \lambda h) \big)
 \end{align*}
 possesses a unique solution $m(\cdot,u,\lambda) \in X$ (close to $m_\eta$) with $m(0) = u$ and $\frac{\partial m}{\partial \nu} =0$  for all $u \in U_\eta$
 and all $\lambda \in V_\eta$. Moreover, the mapping $(u,\lambda) \mapsto m(\cdot,u,\lambda)$ is smooth 
 and $D_u m(T,m_\eta,0) = e^{T \mathcal{L}^\eta}$.
\end{lem}
\begin{proof}
 We define $F:X \times \setR \to Y$ by
\begin{align*}
 F(m,\lambda) = &m_t -\Delta m - \alpha \, m \times \Delta m - \abs{\nabla m}^2 m 
 - \alpha \,
 \eta^2 \, m \times (H[m] + \lambda h) + \eta^2 \, m \times \big( m \times (H[m] + \lambda h) \big)
\end{align*}
and $G:X\times H^2_N \times \setR \to Y \times H^2_N$ by $G(m,u,\lambda) = \big( F(m,\lambda) , m(0) - u\big)$.
Thanks to the embeddings $H^2(\Omega) \hookrightarrow L^\infty(\Omega)$ and $H^1(\Omega) \hookrightarrow L^4(\Omega)$, it is easy to see that $F$ and $G$ are well-defined and smooth. Moreover, we know that $G(m_\eta,m_\eta,0) = (0,0)$ since regular minimizers of $E_{\text{res}}^\eta$ are stationary solutions for LLG with $h_\text{ext}=0$. 

In the sequel we prove that $D_m G(m_\eta,m_\eta,0)$ is invertible. For this we have to show that for all $f \in Y$ and all $u \in H^2_N$ there exists a unique $m \in X$ such that
  $m_t = \mathcal{L}^\eta m + f(t)$ and $m(0)=u$.
Let therefore $f$ and $u$ be given. Since $\mathcal{L}^\eta$ is sectorial, we find a unique mild solution $m$. Thanks to Lemma \ref{lem:optimal regularity}, we obtain $m \in X$, where we have also used that the graph norm of $\mathcal{L}^\eta$ is equivalent to the $H^2$-norm. We conclude that $D_m G(m_\eta,m_\eta,0)$ is invertible. With the help of the implicit function theorem, we find open neighborhoods $U_\eta$ of $m_\eta$ in $H^2_N$, $V_\eta$ of $0$ in $\setR$, and a smooth mapping $m:U_\eta \times V_\eta \to X$ such that $m(\cdot,m_\eta,0)=m_\eta$ and $G\big(m(\cdot,u,\lambda),u,\lambda\big)=(0,0)$ for all $u \in U_\eta$ and $\lambda \in V_\eta$. Hence, $m(\cdot,u,\lambda)$ is the solution we are looking for.

For the remaining statement, we have to calculate the derivative $D_u m(T,m_\eta,0)$ and use the identity $G\big(m(\cdot,u,0),u,0\big)=(0,0)$ to see that
\begin{align*}
 (0,0) = D_m G(m_\eta,m_\eta,0) \circ D_u m (\cdot,m_\eta,0)(h) + D_u G(m_\eta,m_\eta,0)(h)
\end{align*}
for all $h \in H^2_N$. We now set $v = D_u m (\cdot,m_\eta,0)(h)$ for a given $h \in H^2_N$ and find $v_t = \mathcal{L}^\eta v$ and $v(0)=h$.
We conclude $v(t) = e^{t\mathcal{L}^\eta} h$ for all $t \in [0,T]$ and thus $D_u m(T,m_\eta,0) = e^{T\mathcal{L}^\eta}$. The lemma is proved.
\end{proof}
We continue by showing that the pointwise constraint $\abs{m}=1$ is conserved under time evolution:
\begin{lem}
\label{lem:punctual norm}
 Under the assumptions of Lemma \ref{lem:existence}, let $u \in U_\eta$ with $\abs{u}=1$ and $\lambda \in V_\eta$ be given. Then we have
 $\abs{m(t,u,\lambda)}=1$ for all $0\le t \le T$.
\end{lem}
\begin{proof}
 We write $m = m(\cdot,u,\lambda)$ and define $b = \abs{m}^2 -1$. Because of $H^2_N(\Omega) \hookrightarrow L^\infty(\Omega)$, the space $H^2_N(\Omega)$ forms an algebra, and we find
 $b \in C^1([0,T],L^2) \cap C([0,T],H^2_N)$ and $\Delta b = 2 \Delta m \cdot m + 2
 \abs{\nabla m}^2$.
Next, we calculate:
\begin{align*}
 b_t = 2 m_t \cdot m
 = 2 ( \Delta m + \abs{\nabla m}^2 m ) \cdot m
 = \Delta b - 2 \abs{\nabla m}^2 + 2 \abs{\nabla m}^2 \, \abs{m}^2
 = \Delta b + 2 \abs{\nabla m}^2 \, b \, .
\end{align*}
We take the $L^2$-scalar product with $b$ and obtain
\begin{align*}
 \frac{1}{2} \frac{d}{dt} \norm{b}_{L^2}^2 + \norm{\nabla b}_{L^2}^2 = 
 2 \int_\Omega \abs{\nabla m}^2 \,
 \abs{b}^2 \, dx\, .
\end{align*}
We now use the H{\"o}lder inequality with $(3,2,6)$ and find for the right hand side the estimate
\begin{align*}
 2 \int_\Omega \abs{\nabla m}^2 \,\abs{b}^2 \, dx\le 2 \, \norm{\nabla m}_{L^6}^2 \, \norm{b}_{L^4}^2 \le C
 \norm{m}_{H^2}^2 \, \norm{b}_{L^4}^2 \le C \norm{b}_{L^4}^2 \,.
\end{align*}
In order to estimate $\norm{b}_{L^4}^2$ in terms of $\norm{b}_{L^2}^2$ and $\norm{\nabla b}_{L^2}^2$, we make use of the interpolation inequality
 $\norm{u}_{L^4} \le \norm{u}_{L^2}^{\frac{1}{4}} \norm{u}_{L^6}^{\frac{3}{4}}$
for all $u \in L^6(\Omega)$ and the embedding $H^1(\Omega) \hookrightarrow L^6(\Omega)$. We conclude
\begin{align*}
 2 \int_\Omega \abs{\nabla m}^2 \,\abs{b}^2  \, dx\le C
 \norm{b}_{L^2}^{\frac{1}{2}} \, \norm{b}_{L^6}^{\frac{3}{2}} \le C \, \norm{b}_{L^2}^2 + C \, \norm{b}_{L^2}^{\frac{1}{2}} \, \norm{\nabla b}_{L^2}^{\frac{3}{2}} \,.
\end{align*}
With the help of the Young inequality, we obtain
\begin{align*}
 2 \int_\Omega \abs{\nabla m}^2 \,\abs{b}^2 \, dx\le C \, \norm{b}_{L^2}^2 + \frac{1}{2}  \norm{\nabla
 b}_{L^2}^2 \, ,
\end{align*}
thus
\begin{align*}
 \frac{1}{2} \frac{d}{dt} \norm{b}_{L^2}^2 + \frac{1}{2} \norm{\nabla b}_{L^2}^2 \le C \,
 \norm{b}_{L^2}^2 \,.
\end{align*}
Integration leads to
\begin{align*}
 \norm{b(t)}_{L^2}^2 \le C \int_0^t \norm{b(s)}_{L^2}^2 \, ds \quad \text{for all } 0\le t \le T \,,
\end{align*}
where we have used that $b(0)=0$. The Gronwall inequality implies the statement of the lemma.
\end{proof}

\section{Continuation method}
\label{sec:continuation}
From now on, we assume that $h \in C^{0,\beta}\big(\setR,L^2(\Omega,\setR^3)\big)$, where $0<\beta<1$, 
is a $T$-periodic function. As above, we denote by $m(\cdot,u,\lambda)$ the solution of $(LLG)_\eta$ for $u \in H^2_N$ close to $m_\eta$ and $\lambda \in \setR$ close to $0$. Since $m_\eta$ is a stationary solution for $(LLG)_\eta$ with $\lambda = 0$, we already have a $T$-periodic solution $m$ with $\abs{m}=1$. We now ask whether we can find $T$-periodic solutions $m$ of $(LLG)_\eta$ with $\abs{m}=1$ for $\lambda \not=0$. To answer this question, we adapt a well-known method from the theory of ordinary differential equations to our situation. This method -- the so-called continuation method -- relies on the implicit function theorem and has an obvious extension to time-dependent partial differential equations. See for example the book by Amann \cite{Amann} for a description of this method in the setting of ODEs. In contrast to \cite{Amann}, we have a non standard situation at hand. In fact, due to the saturation constraint $\abs{m}=1$, we have to work on an infinite-dimensional manifold and refer to the book by Abraham, Marsden and Ratiu \cite{Abraham} for the notion of infinite-dimensional manifolds modeled on Banach spaces.
\begin{lem}
 \label{lem:manifold}
 The set $\mathcal{M} = \bigset{u \in H^2_N(\Omega,\setR^3) \, \big| \, \abs{u}=1 \text{ pointwise}}$
 defines a smooth submanifold of the space $H^2_N(\Omega,\setR^3)$. For all $u_0 \in \mathcal{M}$, the tangent space $T_{u_0}
 \mathcal{M}$ of $\mathcal{M}$ at $u_0$ is given by
\begin{align*}
 T_{u_0}\mathcal{M} = \bigset{u \in H^2_N(\Omega,\setR^3) \, \big| \, u \cdot u_0 = 0 \text{ pointwise}} \,.
\end{align*}
\end{lem}
\begin{proof}
 We define the map
  $F: H^2_N(\Omega,\setR^3) \to H^2_N(\Omega): u \mapsto \abs{u}^2 -1 $
 and remark that $F$ is well-defined and smooth since $H^2_N(\Omega)$ forms a smooth
 algebra. Moreover, we have that $DF(u)v= 2 \, u \cdot v$ for all $u,v \in H^2_N(\Omega,\setR^3)$. In
 the sequel we show that $0$ is a regular value for $F$. This means we have to show that
 $DF(u)$ is surjective for all $u \in F^{-1}(0)$ and 
 $H^2_N(\Omega,\setR^3) = N(DF(u)) \oplus X(u)$ for all $u \in F^{-1}(0)$, where
  $X(u)$ is a closed subspace of $H^2_N(\Omega,\setR^3)$.
  Let therefore $u \in F^{-1}(0)$ be given. 
  For $f \in H^2_N(\Omega)$ we define 
  $v = 2^{-1} u \, f \in
  H^2_N(\Omega,\setR^3)$.
  Then $DF(u)v = 2 \, u \cdot v = u \cdot u \, f = f$,
  hence $DF(u)$ is surjective. Moreover, $N(DF(u))$ is a closed subspace of
  $H^2_N(\Omega,\setR^3)$, and since $H^2_N(\Omega,\setR^3)$ is a Hilbert space, we find a 
  closed subspace $X(u) \subset H^2_N(\Omega,\setR^3)$ such that $H^2_N(\Omega,\setR^3) =
  N(DF(u)) \oplus X(u)$. This shows that $0$ is a regular value for $F$. The well
  known submersion theorem (see for example \cite[Submersion Theorem 3.5.4]{Abraham}) implies that
  $\mathcal{M}= F^{-1}(0)$ is a smooth submanifold of $H^2_N(\Omega,\setR^3)$ and $T_{u_0}\mathcal{M} =
 N(DF(u_0))$ for all $u_0 \in \mathcal{M}$. The lemma is proved.
\end{proof}
We make use of the following important observation:
\\[1ex]
\centerline{\itshape{
 $m(\cdot,u,\lambda)$ possesses a $T$-periodic extension to $\setR$ that solves $(LLG)_\eta$}}
 \\
 \centerline{\itshape{if and only if $m(T,u,\lambda) = u$.}}
\\[1ex]
In particular, $m=m(\cdot,u,\lambda)$ defines a $T$-periodic solution for $(LLG)_\eta$ with saturation constraint $\abs{m}=1$ if and only if $m(T,u,\lambda)=u$ and $\abs{u}=1$. In order to solve the parameter dependent fixed point equation on $\mathcal{M}$, we need the following version of the implicit function theorem on infinite-dimensional manifolds:
\begin{lem}
\label{lem:continuation}
 Let $\mathcal{N}$ be a smooth manifold and $f:U \times V \to \mathcal{N}$ be a smooth
 mapping, where $U\subset \mathcal{N}$, $V \subset \setR$ are open subsets and $0 \in V$. Suppose that $f(u_0,0) = u_0$ for some $u_0 \in U$ and that
 the operator $D_1 f(u_0,0) - I: T_{u_0} \mathcal{N} \to T_{u_0} \mathcal{N}$
is an isomorphism. Then there exist an open neighborhood $W\subset V$ of $0$ and a smooth mapping $g:W \to U$ such that $g(0)=u_0$ and $f\big(g(\lambda),\lambda) = g(\lambda)$ for all $\lambda \in W$.
\end{lem}
\begin{proof}
 We choose an admissible chart around the point $u_0$. This means that we have an open neighborhood $U_0 \subset U$ of
 $u_0$ in $\mathcal{N}$ together with an $C^\infty$-diffeomorphism $\varphi : U_0 \to \varphi(U_0)$, where
 $\varphi(U_0)$ is an open subset of a Banach space $E$. We can assume that $\varphi (u_0) = 0$, and after shrinking $U_0$ and $V$, we can define
  $F:\varphi(U_0) \times V \to E : (w,\lambda) \mapsto \varphi\big( f( \varphi^{-1}
 (w),\lambda) \big) - w$.
 Then $F$ is a smooth mapping between Banach spaces and $F(0,0)=0$. Moreover, we find with the chain
 rule that
 \begin{align*}
  D_1 F(0,0) = D \varphi (u_0) \circ D_1 f(u_0,0) \circ D \varphi^{-1} (0) - I 
  =  D \varphi (u_0)
 \circ \big(D_1 f(u_0,0) - I\big) \circ D \varphi^{-1} (0) \, ,
 \end{align*}
 hence $D_1 F(0,0)$ is an isomorphism. We can now use the implicit function theorem and
 find an open neighborhood $W \subset V$ of $0$ and a smooth mapping $\overline{g}:W \to \varphi
 (U_0)$ such that
 \begin{align*}
  0= F(\overline{g}(\lambda),\lambda) = \varphi\big( f( \varphi^{-1} \circ
 \overline{g}(\lambda),\lambda) \big) - \overline{g}(\lambda)
 \end{align*}
 for all $\lambda \in W$. Finally, we define $g= \varphi^{-1} \circ \overline{g}$, and the statement
 of the lemma follows.
\end{proof}
In view of Lemma \ref{lem:continuation}, we have to study the invertibility of the linear operator
\begin{align*}
 D_u m(T,m_\eta,0) - I = e^{T\mathcal{L}^\eta} - I
\end{align*}
on the tangent space $T_{m_\eta}\mathcal{M}$. We start by analyzing the restriction of $\mathcal{L}^\eta$ to the tangent space $T_{m_\eta}\mathcal{M}$ and define
\begin{align*}
 \mathcal{D}_\eta = \overline{T_{m_\eta}\mathcal{M}}^{L^2} = \overline{\bigset{u \in
 H^2_N(\Omega,\setR^3) \, \big| \, u \cdot m_\eta = 0}}^{L^2} \,.
\end{align*}
\begin{lem}
 It holds
  $\mathcal{D}_\eta = \bigset{ u \in L^2(\Omega,\setR^3) \, \big| \, u \cdot m_\eta = 0 \text{ almost everywhere}}$,
 and if $u$ belongs to $T_{m_\eta}\mathcal{M}$, then $\mathcal{L}^\eta u$ belongs to $\mathcal{D}_\eta$.
\end{lem}
\begin{proof}
  ``$\subset$'': Let $u \in \mathcal{D}_\eta$ be given. We find a sequence $u_n \in T_{m_\eta}\mathcal{M}$ such that $u_n \to u$ in $L^2$. Moreover, we 
  can assume $u_n \to u$ almost everywhere in \nolinebreak $\Omega$. In particular, we obtain
    $u \cdot {m_\eta} = \lim_{n\to \infty} u_{n} \cdot {m_\eta} = 0$  almost everywhere in $\Omega$.
  \\
  ``$\supset$'': Let $u \in L^2(\Omega, \setR^3)$ with $u \cdot {m_\eta} = 0$ be given. We find a sequence $u_n \in C_0^\infty(\Omega,\setR^3)$ such
  that $u_n \to u$ almost everywhere in $\Omega$ and in $L^2$. We define
    $v_n = u_n - (u_n \cdot {m_\eta}) \, {m_\eta}  \in T_{m_\eta}\mathcal{M}$
  for $n \in \setN$ and find $v_n \to u$ in $L^2$, hence $u \in \mathcal{D}_\eta$.

  For the remaining statement,
  let $u \in T_{m_\eta}\mathcal{M}$ be given. Thanks to what we have just shown, it is enough to 
  check that
  $\mathcal{L}^\eta u \cdot {m_\eta} =0$. From \eqref{eq:linearization} we obtain
  \begin{align*}
    \mathcal{L}^\eta u \cdot {m_\eta} = \Delta u \cdot m_\eta + 2 \, \nabla u : \nabla m_\eta - \eta^2 \, u \cdot H[m_\eta] 
 - \alpha \big( m_\eta \times 
\big(\Delta m_\eta + \eta^2 H[m_\eta]\big) \big)\cdot u \, ,
  \end{align*}
where we have used the vector identities \eqref{eq:vector identities}.
Because of $u \cdot {m_\eta} = 0$, we get
  \begin{align*}
    0 =\Delta ( u \cdot {m_\eta} ) =\Delta u \cdot {m_\eta} + u \cdot \Delta {m_\eta} + 2 \, \nabla{m_\eta} :  \nabla u \, ,
  \end{align*}
  hence
  \begin{align*}
    \mathcal{L}^\eta u \cdot {m_\eta} = - \big( \Delta {m_\eta} + \eta^2 \, H[{m_\eta}] \big) \cdot u -  \alpha \big( m_\eta \times 
\big(\Delta m_\eta + \eta^2 H[m_\eta]\big) \big)\cdot u \,.
  \end{align*}
The fact \eqref{eq:m is parallel to H effective} implies the statement of the lemma.
\end{proof}
Thanks to the previous lemma, we can define the linear operator $\mathcal{L}^\eta_0$ by
\begin{align*}
 \mathcal{L}^\eta_0 : T_{m_\eta}\mathcal{M} \subset \mathcal{D}_\eta \to \mathcal{D}_\eta : u
 \mapsto \mathcal{L}^\eta u \,.
\end{align*}
Moreover, we can show that $\mathcal{L}^\eta_0$ is sectorial:
\begin{lem}
\label{lem:L0 is sectorial}
 The linear operator $\mathcal{L}^\eta_0$ is sectorial and $e^{t \mathcal{L}^\eta_0} = {e^{t
 \mathcal{L}^\eta}}_{|\mathcal{D}_\eta}$.
\end{lem}
\begin{proof}
 First of all, we prove the existence of a sector $S_{\theta,\omega} \subset \rho (\mathcal{L}^\eta)$ as in Definition \ref{def:sectorial} such that
\begin{align*}
 R(\lambda,\mathcal{L}^\eta) (D_\eta) \subset T_{m_\eta} \mathcal{M}
\end{align*}
for all $\lambda \in S_{\theta,\omega}$. For this we use that
 $\mathcal{A} : H^2_N(\Omega) \subset L^2(\Omega) \to L^2(\Omega): v \mapsto \Delta v 
 + 2 \, \abs{\nabla m_\eta}^2 v$
defines a sectorial operator ($\mathcal{A} =$ Laplacian $+$ \textquotedblleft{}lower order term\textquotedblright{}). Since $\mathcal{L}^\eta$ is a sectorial operator as well, we can choose a sector $S_{\theta,\omega} \subset \rho(\mathcal{A}) \cap \rho(\mathcal{L}^\eta)$. Let now $\lambda \in S_{\theta,\omega}$ and $f \in
\mathcal{D}_\eta$ be given. We have to show that $u = R(\lambda,\mathcal{L}^\eta) f $ belongs to 
$T_{m_\eta}\mathcal{M}$. The definition of $u$ implies $\lambda u - \mathcal{L}^\eta u = f$, and multiplication with $m_\eta$ leads to
$\lambda u \cdot m_\eta - \mathcal{L}^\eta u \cdot m_\eta = 0$.
With the help of \eqref{eq:linearization}, we get
\begin{align*}
 \mathcal{L}^\eta u \cdot m_\eta =& \Delta u \cdot m_\eta + 2 \, \nabla u : \nabla m_\eta +\abs{\nabla m_\eta}^2 u \cdot m_\eta 
 - \eta^2 \big( u \times \big( m_\eta \times H[m_\eta] \big)
 \big) \cdot m_\eta  
 \\
 &+ \alpha ( u \times \Delta m_\eta ) \cdot m_\eta 
 + \alpha \, \eta^2 \big( u \times H[m_\eta] \big) \cdot m_\eta \,.
\end{align*}
Moreover, we use the identity $\,\Delta (u \cdot m_\eta) = \Delta u \cdot m_\eta + u \cdot \Delta m_\eta + 2 \, \nabla u : \nabla
 m_\eta\,$
and obtain
\begin{align*}
 \mathcal{L}^\eta u \cdot m_\eta =& \Delta (u \cdot m_\eta) + \abs{\nabla m_\eta}^2 u \cdot m_\eta 
 - u \cdot \Delta m_\eta + \eta^2 \big( m_\eta \times \big( m_\eta \times H[m_\eta] \big)
 \big) \cdot u  
 \\
 &- \alpha ( m_\eta \times \Delta m_\eta ) \cdot u - \alpha \, \eta^2 \big( m_\eta \times H[m_\eta] \big) \cdot u \,.
\end{align*}
From here we get
\begin{align*}
 \mathcal{L}^\eta u \cdot m_\eta =& \Delta (u \cdot m_\eta) + 2 \, \abs{\nabla m_\eta}^2 u 
 \cdot m_\eta + \big( m_\eta \times \big( m_\eta \times \big( \Delta m_\eta + \eta^2 H[m_\eta] \big) \big) \big) \cdot u 
\\
&- \alpha \big( m_\eta \times \big( \Delta m_\eta + \eta^2 H[m_\eta] \big) \big) \cdot u
\\
=& \Delta (u \cdot m_\eta) + 2 \, \abs{\nabla m_\eta}^2 u \cdot m_\eta \,,
\end{align*}
where we have used the identity \eqref{eq:identity for Laplacian} and the fact \eqref{eq:m is parallel to H effective}.
In particular, we have
 $\lambda (u \cdot m_\eta) - \mathcal{A}(u \cdot m_\eta) = 0$,
hence $u \cdot m_\eta =0$ and $u \in T_{m_\eta} \mathcal{M}$.

From what we have shown above, we obtain $S_{\theta,\omega} \subset \rho(\mathcal{L}^\eta_0)$ and $R(\lambda,\mathcal{L}^\eta_0) = R(\lambda,\mathcal{L}^\eta)_{|\mathcal{D}_\eta}$ for all $\lambda \in S_{\theta,\omega}$. The required resolvent estimate is also satisfied, and we conclude that $\mathcal{L}^\eta_0$ is sectorial and $e^{t \mathcal{L}^\eta_0} = {e^{t
 \mathcal{L}^\eta}}_{|\mathcal{D}_\eta}$. The lemma is proved.
\end{proof}
We use Lemma \ref{lem:L0 is sectorial} in the following way: Assume for the moment that 
$e^{T \mathcal{L}^\eta_0} - I : \mathcal{D}_\eta \to \mathcal{D}_\eta$
 is invertible. Then we also have that $e^{T \mathcal{L}^\eta} - I : T_{m_\eta}\mathcal{M} \to T_{m_\eta}\mathcal{M}$ is invertible. Indeed, since 
$e^{T \mathcal{L}^\eta_0} = {e^{T\mathcal{L}^\eta}}_{|\mathcal{D}_\eta}$, we 
see that $e^{T \mathcal{L}^\eta} - I : T_{m_\eta}\mathcal{M} \to T_{m_\eta}\mathcal{M}$ is injective. Moreover, for $f \in T_{m_\eta}\mathcal{M}$ we find a unique
$u \in \mathcal{D}_\eta$ such that $e^{T \mathcal{L}^\eta} u - u = f$. Thanks to the smoothing property of $e^{T \mathcal{L}^\eta}$, we obtain
 $u = e^{T \mathcal{L}^\eta} u - f \in T_{m_\eta} \mathcal{M}$,
hence $e^{T \mathcal{L}^\eta} - I : T_{m_\eta}\mathcal{M} \to T_{m_\eta}\mathcal{M}$ is surjective.

Working with $e^{T \mathcal{L}^\eta_0}$ on $\mathcal{D}_\eta$ has the 
advantage that we can use the spectral mapping theorem for sectorial 
operators (see for example \cite[Corollary 2.3.7]{lunardi}). In particular we have that
\begin{align*}
  1 \not\in \sigma(e^{T \mathcal{L}^\eta_0}) \quad \Leftrightarrow \quad 
  \frac{2 k \pi \dot{\imath}}{T} \not\in \sigma(\mathcal{L}^\eta_0) 
\text{ for all } k \in \setZ\,.
\end{align*}
Furthermore, we know that the resolvents of $\mathcal{L}^\eta_0$ are compact due to the compact Sobolev embedding $H^2_N(\Omega) \hookrightarrow L^2(\Omega)$.
This shows that the spectrum $\sigma(\mathcal{L}^\eta_0)$ consists entirely of isolated eigenvalues with finite-dimensional eigenspaces (see for example \cite[Theorem 6.29]{Kato}), hence $\sigma(\mathcal{L}^\eta_0) = \sigma_P(\mathcal{L}^\eta_0)$.
In particular, for 
\begin{align*}
e^{T \mathcal{L}^\eta} - I : T_{m_\eta}\mathcal{M} \to T_{m_\eta}\mathcal{M}
\end{align*}
 being invertible, it is enough to check that $\sigma_P(\mathcal{L}^\eta_0) \cap \dot{\imath} \setR = \emptyset$.
In the subsequent sections, we prove the validity of this statement in two steps. First, we 
show that $\dot{\imath}t \not\in \sigma_P(\mathcal{L}^\eta_0)$ for every 
$t \in \setR \setminus \set{0}$, and then we prove that 
$0 \not \in \sigma_P(\mathcal{L}^\eta_0)$. It turns out that the latter statement requires a restriction on the shape of $\Omega$, and our particle -- or
equivalently $\eta$ -- has to be sufficiently small.

\section{Spectral analysis for $\mathcal{L}^\eta_0$ - First step}
\label{sec:spectral analysis - first step}
In this section we prove the following lemma:
\begin{lem}
\label{lem:first step}
 The linear operator $\mathcal{L}^\eta_0$ satisfies 
 $\sigma_P(\mathcal{L}^\eta_0) \cap \dot{\imath} \setR \setminus \set{0} = \emptyset$.
\end{lem}
\begin{proof}
 Let $t \in \setR \setminus \set{0}$ be given. We have to show that $\,\dot{\imath}t -\mathcal{L}^\eta_0 : T_{m_\eta} \mathcal{M} \subset \mathcal{D}_\eta \to \mathcal{D}_\eta\,$
 is injective. Let therefore $u \in T_{m_\eta} \mathcal{M}$ be such that $0 = \dot{\imath}t u - \mathcal{L}^\eta_0 u$.
 We obtain
\begin{align*}
 0 &= ( \alpha m_\eta \times \cdot - I ) (\dot{\imath}t u - \mathcal{L}^\eta_0 u ) 
\end{align*}
and rewrite $\mathcal{L}^\eta_0 u$ with the help of \eqref{eq:vector identities} as follows:
\begin{equation}
\label{eq:identity for L0eta}
 \begin{split}
 \mathcal{L}^\eta_0 u =& \Delta u + \alpha \, m_\eta \times \Delta u 
+ 2 \, \nabla u : \nabla m_\eta \, m_\eta + \abs{\nabla m_\eta}^2 u  
+ \alpha \, u \times \Delta m_\eta 
+ \alpha \, \eta^2 \, m_\eta \times H[u] 
\\
&+ \alpha \, \eta^2 \, u \times H[m_\eta] 
- \eta^2 \, m_\eta \cdot H[u] \, m_\eta + \eta^2 H[u]
- \eta^2 m_\eta \cdot H[m_\eta] \, u 
- \eta^2  u \cdot H[m_\eta] \, m_\eta \,.
\end{split}
\end{equation}
Next, we calculate:
\begin{align*}
 \alpha m_\eta \times \mathcal{L}^\eta_0 u =& \alpha m_\eta \times \Delta u + \alpha^2 m_\eta \cdot
 \Delta u \, m_\eta - \alpha^2 \Delta u + \alpha \, \abs{\nabla m_\eta}^2 m_\eta \times u
 + \alpha^2 m_\eta \cdot \Delta m_\eta \, u 
 \\
 &+ \alpha^2 \eta^2 m_\eta \cdot H[u] \, m_\eta 
 -\alpha^2 \eta^2 H[u] + \alpha^2 \eta^2 m_\eta \cdot H[m_\eta] \, u + \alpha \eta^2 m_\eta \times H[u]
 \\
 &- \alpha \eta^2
 m_\eta \cdot H[m_\eta] \, m_\eta \times u \,.
\end{align*}
Furthermore, we find:
\begin{align*}
 \mathcal{L}^\eta_0 u- \alpha m_\eta \times \mathcal{L}^\eta_0 u
 =& (1+\alpha^2) \Delta u + (1+\alpha^2) \eta^2 H[u] - (1+\alpha^2) \eta^2 m_\eta \cdot H[m_\eta] \,
 u - (1+\alpha^2) \eta^2 m_\eta \cdot H[u] \, m_\eta 
 \\
 &- \alpha^2 m_\eta \cdot \Delta u \, m_\eta
 - \alpha^2 m_\eta \cdot \Delta m_\eta \, u
 + 2 \, \nabla u : \nabla m_\eta \, m_\eta + \abs{\nabla m_\eta}^2 \, u - \eta^2 u \cdot H[m_\eta]
 \, m_\eta
 \\
 &- \alpha \big( \Delta m_\eta + \abs{\nabla m_\eta}^2 \, m_\eta - \eta^2 m_\eta \cdot H[m_\eta] \, m_\eta 
 + \eta^2 H[m_\eta] \big) \times u \,.
\end{align*}
We observe that the last term is equal to zero thanks to the Euler-Lagrange equation \eqref{eq:Euler-Lagrange}. 
We now take the $L^2$-scalar product with $u$ and obtain
\begin{align*}
 \big( \mathcal{L}^\eta_0 u- \alpha m_\eta \times \mathcal{L}^\eta_0 u,u
\big)_{L^2} 
=& -(1+\alpha^2) \int_\Omega \abs{\nabla u}^2 \, dx+ (1+\alpha^2) \eta^2 \int_\Omega H[u] \cdot \overline{u} \, dx
 \\
 &- (1+\alpha^2) \eta^2 \hspace{-0.15cm}\int_\Omega \hspace{-0.1cm} m_\eta \cdot H[m_\eta] \, \abs{u}^2  dx- \alpha^2 \hspace{-0.15cm}\int_\Omega \hspace{-0.1cm} m_\eta \cdot \Delta
 m_\eta \, \abs{u}^2  dx+ \hspace{-0.15cm}\int_\Omega \abs{\nabla m_\eta}^2 \, \abs{u}^2  dx. 
\end{align*}
Since $H$ is $L^2$-symmetric (see Lemma \ref{lem:stray field}), we conclude $\big( \mathcal{L}^\eta_0 u- \alpha m_\eta \times \mathcal{L}^\eta_0 u,u\big)_{L^2}  \in \setR$.
We now decompose $u = u_1 + \dot{\imath} u_2$ into real and imaginary part and calculate:
\begin{align*}
 (\alpha m_\eta \times \cdot - I ) \, \dot{\imath} t u = \dot{\imath} \alpha t m_\eta \times u_1 - \alpha t m_\eta \times u_2 - \dot{\imath} t u_1 + t u_2 \,.
\end{align*}
Again, we take the $L^2$-scalar product with $u$ and find
\begin{align*}
 \big((\alpha m_\eta \times \cdot - I ) \, \dot{\imath} t u, u \big)_{L^2}  = -\dot{\imath} t \int_\Omega \abs{u}^2 \, dx
 + 2 \alpha t \int_\Omega  (m_\eta \times u_1) \cdot u_2 \, dx\,.
\end{align*}
In particular, we obtain
\begin{align*}
 0 = \text{Im} \big( ( \alpha m_\eta \times \cdot - I ) (\dot{\imath}t u - \mathcal{L}^\eta_0 u ),u \big)_{L^2} 
 = - \dot{\imath} t \int_\Omega \abs{u}^2 \, dx\, ,
\end{align*}
hence $u=0$. The lemma is proved.
\end{proof}

\section{Spectral analysis for $\mathcal{L}^\eta_0$ - Second step}
\label{sec:spectral analysis - second step}
In general we can not expect that $0 \not \in \sigma_{P}(\mathcal{L}^\eta_0)$ since symmetries of the domain \nolinebreak $\Omega$ lead naturally to a nontrivial kernel of $\mathcal{L}^\eta_0$. Assume for example that we can find a smooth path $R: \setR \to SO(3)$ such that $R(0)=I$ and $R(t)(\Omega) = \Omega$ for all $t \in \setR$. Then $m(t)$ defined by $m(t) = R(t) \circ m_\eta \circ R(t)^{T}$
is also a minimizer of $E_{\text{res}}^\eta$ for all $t\in \setR$. In particular, $m(t)$ is a stationary solution for the rescaled LLG with 
$h_\text{ext}=0$, and differentiation shows $w=\frac{d}{dt} m(t)_{|t=0} \in T_{m_\eta}\mathcal{M}$ and $\mathcal{L}^\eta_0 w =0$.
If $w \not=0$, then we conclude that $0 \in \sigma_{P}(\mathcal{L}^\eta_0)$. We also remark that, due to the non-local stray field $H$, it might be difficult to investigate the dimension of the kernel of $\mathcal{L}^\eta_0$, even for symmetric domains $\Omega$.

In order to rule out nontrivial zeros of $\mathcal{L}^\eta_0$, we need a restriction on the shape of $\Omega$. We can conveniently capture the relation between geometry and non-local stray field with the help of the demagnetizing tensor $\dt$ defined by
\begin{align*}
 \dt : \setR^3 \to \setR^3 : u \mapsto - \int_\Omega H[u] \, dx\,.
\end{align*}
From the definition we see that $\dt$ is linear, symmetric, and positive definite, and we can therefore write without loss of generality $\dt$ in diagonal form $\dt = \diag(\lambda_1,\lambda_2,\lambda_3)$
with positive eigenvalues $\lambda_1,\lambda_2,\lambda_3$. In the remaining chapter, we assume that the smallest eigenvalue $\lambda_1$ of $\dt$ is simple, that is $\lambda_1 < \lambda_2 \le \lambda_3$,
which, roughly speaking, means that the length of $\Omega$ is greater than its height and its width. For rotation ellipsoids, explicit formulas for the demagnetizing tensor are available (see \cite[Section 3.2.5]{HS}), and we see that the set of bounded $C^{2,1}$-domains with the property $ \lambda_1 < \lambda_2 \le \lambda_3$ is not empty.
The magnetic shape anisotropy as expressed above keeps minimizers  $m_\eta$  
of \nolinebreak $E_{\text{res}}^\eta$ close to the $e_1$-axis, provided $\eta$ is small enough. Before we prove this statement, we need a refined $L^2$-estimate for the gradient of $m_\eta$ and start with some notations:

We decompose an arbitrary function $u \in H^2_N(\Omega,\setR^3)$ by $u = \ui + \uii$, where
\begin{align}
 \label{eq:decomposition}
 \ui = \int_\Omega u \, dx\qquad \text{and} \qquad \uii = u - \int_\Omega u \, dx\,.
\end{align}
From the definition we find
\begin{equation}
 \label{eq:integral of uii}
 \int_\Omega \uii \, dx= 0 \qquad \text{and} \qquad \int_\Omega \ui \cdot \uii \, dx= 0 \,.
\end{equation}
This in particular implies that
\begin{equation}
 \label{eq:estimate for uii}
 \norm{\uii}_{L^2} \le C \, \norm{\nabla \uii}_{L^2} = C \, \norm{\nabla u}_{L^2}
\end{equation}
 and
 \begin{equation}
 \label{eq:estimate for u}
 \norm{u}_{L^2}^2 = \abs{\ui}^2 + \norm{\uii}_{L^2}^2 \le \abs{\ui}^2 
 + C \, \norm{\nabla u}_{L^2}^2 \,.
\end{equation}
Furthermore, we write $u = (u_1,u_2,u_3)$ and $\ui = (\vi,\vii,\viii)$.
\begin{lem}
  \label{lem:l2estimate}
  There exist positive constants $\eta_0=\eta_0(\Omega)$ and 
  $C_0=C_0(\Omega)$ such that $\norm{\nabla m_\eta}_{L^2} \le C_0  \eta^2$
  for every minimizer $m_\eta \in H^1(\Omega,S^2)$ of $E_{\text{res}}^\eta$ whenever $0<\eta\le\eta_0$.
\end{lem}
\begin{proof}
 Let $m_\eta = \metai + \metaii$ be a minimizer of $E_{\text{res}}^\eta$. The saturation constraint $\abs{m_\eta}=1$ implies
  \begin{align*}
    1 = \abs{\metai}^2 + 2 \, \metai \cdot \metaii + \abs{\metaii}^2 \, ,
  \end{align*}
  and by integration we obtain together with \eqref{eq:integral of uii} the identity
  \begin{equation}
   \label{eq:identity for ui and uii}
    1 = \abs{\metai}^2 + \int_\Omega \abs{\metaii}^2 \, dx\,.
  \end{equation}
  With the help of the \Poincare{} inequality and \eqref{eq:regularity}, we see 
  that $\abs{1 - \abs{\metai}^2} \le C  \eta^2$
  for $\eta$ small enough.
  Therefore, we can introduce the constant comparison function
    $v = \metai/\abs{\metai} \in H^1(\Omega,S^2)$.
  Since $m_\eta$ is a minimizer, we find $E_{\text{res}}^\eta (m_\eta) \le E_{\text{res}}^\eta (v)$
  and rewrite this inequality as follows:
  \begin{align*}
    \int_\Omega \abs{\nabla m_\eta}^2 \, dx\le \eta^2 \int_{\setR^3} \Bigabs{ H\Big[\frac{\metai}{\abs{\metai}}\Big]}^2 \, dx
    - \eta^2 \int_{\setR^3} \abs{H[m_\eta]}^2 \, dx
   = \eta^2 \int_{\setR^3} H\Big[\frac{\metai}{\abs{\metai}} - m_\eta\Big] \cdot H\Big[ \frac{\metai}{\abs{\metai}} + m_\eta\Big] \, dx\,.
  \end{align*}
  With the help of the H{\"o}lder inequality, we obtain
  \begin{align*}
    \int_\Omega \abs{\nabla m_\eta}^2 \, dx& \le \eta^2 \Bignorm{H\Big[\frac{\metai}{\abs{\metai}} - m_\eta\Big]}_{L^2} \, \Bignorm{H\Big[\frac{\metai}{\abs{\metai}} 
      + m_\eta\Big]}_{L^2}
     \le 2 \eta^2 \Bignorm{\frac{\metai}{\abs{\metai}} - m_\eta}_{L^2} \,.
  \end{align*}
  Moreover, we have the identity
  \begin{align*}
    \Bignorm{\frac{\metai}{\abs{\metai}} - \metai}_{L^2}^2 &= \int_\Omega \Bigabs{\frac{\metai}{\abs{\metai}} - \metai}^2 \, dx
    = \int_\Omega \big( 1 - \abs{\metai} \big)^2 \, dx\, ,
  \end{align*}
  and since $\abs{\metai} \le 1$, we can estimate as follows:
  \begin{align*}
    \big( 1 - \abs{\metai} \big)^2 = 1-2 \, \abs{\metai} + \abs{\metai}^2 \le 1 - 2 \, \abs{\metai}^2 + \abs{\metai}^2 = 1 - \abs{\metai}^2 \,.
  \end{align*}
  This together with the \Poincare{} inequality shows that
  \begin{align*}
    \Bignorm{\frac{\metai}{\abs{\metai}} - \metai}_{L^2}^2  \le \int_\Omega \big(1 - \abs{\metai}^2\big) \, dx= \int_\Omega \abs{m_\eta - \metai}^2\, dx
     \le C \int_\Omega \abs{\nabla m_\eta}^2 \, dx\, .
  \end{align*}
 We conclude
  \begin{align*}
    \norm{\nabla m_\eta}_{L^2}^2 & \le 2  \eta^2 \Bignorm{\frac{\metai}{\abs{\metai}} - \metai}_{L^2} + 2  \eta^2 \norm{ \metai - m_\eta }_{L^2}
     \le C  \eta^2 \norm{\nabla m_\eta}_{L^2} + C  \eta^2 \norm{\nabla m_\eta}_{L^2}
     = C \eta^2 \norm{\nabla m_\eta}_{L^2} \, .
  \end{align*}
 The lemma is proved.
\end{proof}
By interpolation between $L^2$ and $L^\infty$, we obtain from Lemma \ref{lem:l2estimate} and \eqref{eq:regularity} the following corollary:
\begin{corollary}
  \label{cor:l4estimate}
  There are positive constants $\eta_0=\eta_0(\Omega)$ and $C_0=C_0(\Omega)$ with the following property: For every $2 \le p \le \infty$, we have the estimate
    $\norm{\nabla m_\eta}_{L^p} \le C_0  \eta^{1+\frac{2}{p}}$
  whenever $m_\eta \in H^1(\Omega,S^2)$ is a minimizer of $E_{\text{res}}^\eta$ with parameter $0< \eta \le \eta_0$.
\end{corollary}
We can now show that minimizers of $E_{\text{res}}^\eta$ stay close to the $e_1$-axis, provided the parameter $\eta$ is small enough.
\begin{lem}
  \label{lem:orientation}
  There exist positive constants $\eta_0=\eta_0(\Omega)$ and 
  $C_0=C_0(\Omega)$ such that we either have $\norm{m_\eta-e_1}_{L^\infty} \le C_0 \eta$ or $\norm{m_\eta+e_1}_{L^\infty} \le C_0 \eta$
  for every minimizer $m_\eta \in H^1(\Omega,S^2)$ of the rescaled energy functional $E_{\text{res}}^\eta$ whenever $0< \eta \le \eta_0$.
\end{lem}
\begin{proof}
 Let $m_\eta$ be a minimizer of $E_{\text{res}}^\eta$. For convenience we write $m= m_\eta$ and use
 again the decomposition $m = \mi + \mii$ from \eqref{eq:decomposition}.
 We choose the constant comparison function $e_1 \in H^1(\Omega,S^2)$ and employ the definition of \nolinebreak $\dt$ combined with Lemma \ref{lem:stray field} to find that
 \begin{align*}
  \int_{\setR^3} \abs{H[m]}^2 \, dx\le \frac{1}{\eta^2} E_{\text{res}}^\eta(m) \le \frac{1}{\eta^2}
  E_{\text{res}}^\eta(e_1) = - \int_\Omega H[e_1] \cdot e_1 \, dx= \dt e_1 \cdot
  e_1 = \lambda_1 \,.
 \end{align*}
  It follows that
  \begin{align*}
    \lambda_1 
     &\ge - \int_\Omega H[\mi]\cdot \mi \, dx- \int_\Omega H[\mi] \cdot \mii \, dx- \int_\Omega H[\mii] \cdot \mi \, dx- \int_\Omega H[\mii] \cdot \mii\, dx
    \\
    & = \dt \mi \cdot \mi - 2 \int_\Omega H[\mi] \cdot \mii \, dx+ \int_{\setR^3} \abs{H[\mii]}^2\, dx
    \\
    &\ge \dt \mi \cdot \mi - 2 \int_\Omega H[\mi] \cdot \mii \, dx\,.
  \end{align*}
  From here we obtain
  \begin{align*}
    \lambda_1 &\ge \lambda_1 \wi^2 + \lambda_2 \wii^2 + \lambda_3 \wiii^2 - 2 \int_\Omega H[\mi] \cdot \mii\, dx
    \ge \lambda_1 \wi^2 + \lambda_2 ( \wii^2 + \wiii^2 ) - 2 \int_\Omega H[\mi] \cdot \mii\, dx
  \end{align*}
  because of $\lambda_3 \ge \lambda_2$. Due to the saturation constraint $\abs{m}=1$, we can use \nolinebreak \eqref{eq:identity for ui and uii}, and therefore we find
  \begin{align*}
    \lambda_1 &\ge (\lambda_1 - \lambda_2) \wi^2 + \lambda_2 - \lambda_2 \int_\Omega \abs{\mii}^2 \, dx- 2 \int_\Omega H[\mi] \cdot \mii \, dx\,.
  \end{align*}
  With Lemma \ref{lem:l2estimate} and the \Poincare{} inequality, we obtain the estimate
  \begin{align*}
    (\lambda_2 - \lambda_1 ) \, \abs{1- \wi^2} &\le \lambda_2 \int_\Omega \abs{\mii}^2 \, dx+ 2 \int_\Omega H[\mi] \cdot \mii\, dx
     \le C \norm{\nabla m}_{L^2}^2 + 2 \norm{H[\mi]}_{L^2} \, \norm{\mii}_{L^2}
     \le C  \eta^2 \,.
  \end{align*}
  Since $\lambda_1 < \lambda_2$, we conclude
  $\abs{1- \wi^2} \le C  \eta^2$ for $\eta$ small enough. We now have to distinguish the cases $\wi \ge 0$ and $\wi < 0$.
  If $\wi \ge 0$, then we get
  \begin{align*}
    \abs{\mi - e_1}^2 &= ( 1 - \wi)^2 + \wii^2 + \wiii^2
    \le  1 - \wi^2 + 1 - \wi^2 - \int_\Omega \abs{\mii}^2\, dx
    \le C \eta^2 \,.
  \end{align*}
   We make use of the Sobolev embedding $W^{1,6}(\Omega,\setR^3) \hookrightarrow
   L^\infty(\Omega,\setR^3)$, the \Poincare{} inequality, and \eqref{eq:regularity} to see that
  \begin{align*}
   \norm{m - \mi}_{L^\infty} &\le C \norm{m - \mi}_{W^{1,6}}
  \le C \big( \norm{m- \mi}_{L^6} + \norm{\nabla m}_{L^6} \big)
  \le C \norm{\nabla m}_{L^6}
  \le C \eta
  \end{align*}
  for $\eta$ small enough. We now conclude
  \begin{align*}
    \norm{m - e_1}_{L^\infty} &\le \norm{m - \mi}_{L^\infty} + \abs{\mi - e_1} \le C_0 \eta \,.
  \end{align*}
  Similarly, we find $ \norm{m + e_1}_{L^\infty} \le  C_0 \eta\,$ if $\,\wi < 0$. The lemma is proved.
\end{proof}
The statement $0 \not\in \sigma_{P}(\mathcal{L}^\eta_0)$ is a consequence of the estimates established in the next two lemmas. In the sequel we assume without loss of generality that $m_\eta = (m^\eta_1,m^\eta_2,m^\eta_3)$ satisfies
\begin{equation}
 \label{eq:assumptions for m}
 \norm{\nabla m_\eta}_{L^p} \le C_0 \eta^{1+\frac{2}{p}} \qquad \text{and} \qquad 
 \norm{m_\eta - e_1}_{L^\infty} \le C_0 \eta
\end{equation}
for $2 \le p \le \infty$ and $\eta$ small enough (replace $m_\eta$ by $-m_\eta$ if necessary). We now state the first estimate:
\begin{lem}
 \label{lem:first estimate}
 There exist positive constants $\eta_0=\eta_0(\Omega)$ and $C=C(\Omega)$ such that
 \begin{align*}
  (-\mathcal{L}^\eta_0 u,u)_{L^2}  \ge &(1 - C \eta 
  - C \abs{\alpha} \eta^{\frac{1}{2}}) \norm{\nabla u}_{L^2}^2
 + \eta^2 \big( (\lambda_2 - \lambda_1) \abs{\ui}^2 - \alpha (\lambda_3 - \lambda_2)\vii \viii 
 - C \eta \abs{\ui}^2 - C \abs{\alpha} \eta^{\frac{1}{2}} \abs{\ui}^2 \big)
 \end{align*}
 for every $u \in T_{m_\eta}\mathcal{M}$ and every $0<\eta\le\eta_0$.
\end{lem}
\begin{proof}
 Let $u \in T_{m_\eta}\mathcal{M}$ be given. 
 For convenience we write $\int = \int_\Omega$ and find with the help of \eqref{eq:identity for L0eta} the identity
\begin{align*}
 (-\mathcal{L}^\eta_0 u, u)_{L^2}  =& \int \abs{\nabla u}^2 \, dx- \alpha \int (m_\eta \times
 \Delta u) \cdot u \, dx- \int \abs{\nabla m_\eta}^2 \abs{u}^2 \, dx
 - \alpha \, \eta^2 \int ( m_\eta \times H[u] ) \cdot u \, dx
 \\
 &- \eta^2 \int H[u] \cdot u \, dx
 + \eta^2 \int m_\eta \cdot H[m_\eta] \abs{u}^2\, dx
\\
=& \norm{\nabla u}_{L^2}^2 - I_1 - I_2 - I_3 - I_4 - I_5 \,.
\end{align*}
In the following we analyze each term separately. We remark that $\norm{\nabla u}_{L^2}^2$ and $I_4$ are our good terms, and because of $I_3$, we have to do some extra work in form of Lemma \ref{lem:second estimate}. For $I_1$ we get by integration by parts that
\begin{align*}
 I_1 
 =& \alpha \int u_1 ( \nabla m^\eta_3 \cdot \nabla u_2 - \nabla m^\eta_2 \cdot \nabla u_3) \, dx
  + \alpha \int u_2 ( \nabla m^\eta_1 \cdot \nabla u_3 - \nabla m^\eta_3 \cdot \nabla u_1)\, dx
 \\
 &+ \alpha \int u_3 ( \nabla m^\eta_2 \cdot \nabla u_1 - \nabla m^\eta_1 \cdot \nabla u_2) \, dx\,.
\end{align*}
With the help of the H{\"o}lder inequality, the embedding $H^1(\Omega) \hookrightarrow L^4(\Omega)$, \eqref{eq:estimate for u}, and 
\eqref{eq:assumptions for m}, we obtain the estimate
\begin{align*}
 I_1 &\le C \abs{\alpha} \,\norm{\nabla m_\eta}_{L^4} \norm{u}_{L^4} \norm{\nabla u}_{L^2}
 \le C \abs{\alpha} \,\eta^{1+\frac{1}{4}} \abs{\ui} \,\eta^{\frac{1}{4}} \norm{\nabla u}_{L^2} 
 + C \abs{\alpha} \,\eta^{1+\frac{1}{2}} \norm{\nabla u}_{L^2}^2 \,.
\end{align*}
We now use the Young inequality and end up with
\begin{align*}
 I_1 &\le C \abs{\alpha} \,\eta^{2+\frac{1}{2}} \abs{\ui}^2 +  C \abs{\alpha} \,\eta^{\frac{1}{2}} 
 \norm{\nabla u}_{L^2}^2 + C \abs{\alpha} \,\eta^{1+\frac{1}{2}} \norm{\nabla u}_{L^2}^2
 \le C \abs{\alpha} \,\eta^{2+\frac{1}{2}} \abs{\ui}^2 + C \abs{\alpha} \,\eta^{\frac{1}{2}} 
 \norm{\nabla u}_{L^2}^2 \,.
\end{align*}
Similarly, we can estimate $I_2$ as follows:
\begin{align*}
 I_2 
 \le \norm{\nabla m_\eta}_{L^4}^2 \norm{u}_{L^4}^2
 \le C \eta^3 ( \norm{u}_{L^2}^2 + \norm{\nabla u}_{L^2}^2)
 \le C \eta^3 \abs{\ui}^2 + C \eta^3 \norm{\nabla u}_{L^2}^2 \,.
\end{align*}
Next, we split $I_3$ into two terms:
\begin{align*}
 I_3 &= \alpha \, \eta^2 \int (m_\eta \times H[u]) \cdot u \, dx
= \alpha \, \eta^2 \int \big( (m_\eta - e_1) \times H[u] \big) \cdot u \, dx+ \alpha \, \eta^2 \int (e_1 \times H[u]) \cdot u\, dx
= I_3^1 + I_3^2 \,.
\end{align*}
For the first term, we find with \eqref{eq:estimate for u} and \eqref{eq:assumptions for m} the estimate
\begin{align*}
 I_3^1 &\le \abs{\alpha} \, \eta^2 \norm{m_\eta - e_1}_{L^\infty} \norm{H[u]}_{L^2} \norm{u}_{L^2}
 \le C \abs{\alpha} \, \eta^3 \norm{u}_{L^2}^2
 \le C \abs{\alpha} \, \eta^3 \abs{\ui}^2 + C \abs{\alpha} \, \eta^3 \norm{\nabla u}_{L^2}^2 \,.
\end{align*}
For the second term, we use the decomposition $u = \ui + \uii$ from \eqref{eq:decomposition} and obtain:
\begin{align*}
 I_3^2 
 =& \alpha \, \eta^2 \int ( e_1 \times H[\ui] ) \cdot \ui \, dx
   + \alpha \, \eta^2 \int (e_1 \times H[\ui]) \cdot \uii \, dx
 + \alpha \, \eta^2 \int (e_1 \times H[\uii] ) \cdot \ui \, dx
 \\  
&+ \alpha \, \eta^2 \int ( e_1 \times H[\uii]) \cdot \uii \, dx.
\end{align*}
The definition of the demagnetizing tensor $\dt$, the H{\"o}lder inequality, and the estimate \eqref{eq:estimate for uii} imply
\begin{align*}
 I_3^2 
 &\le -\alpha \, \eta^2 (e_1 \times \dt \ui ) \cdot \ui + C \abs{\alpha} 
 \, \eta^{1+\frac{1}{2}} \abs{\ui}
 \, \eta^{\frac{1}{2}} \norm{\nabla u}_{L^2} + C \abs{\alpha} \, \eta^2 \norm{\nabla u}_{L^2}^2 \,.
\end{align*}
Again, the Young inequality leads to the estimate
\begin{align*}
 I_3^2 &\le -\alpha \, \eta^2 (e_1 \times \dt \ui ) \cdot \ui + C \abs{\alpha} \,\eta^3 \abs{\ui}^2 
 + C \abs{\alpha} \,\eta \norm{\nabla u}_{L^2}^2 \,.
\end{align*}
We observe that $e_1 \times \dt\ui = (0, -\lambda_3 \viii, \lambda_2 \vii)$
hence
$(e_1 \times \dt \ui) \cdot \ui = -(\lambda_3 -\lambda_2) \vii \viii$.
We conclude
\begin{align*}
 I_3 \le \alpha \, \eta^2 ( \lambda_3 - \lambda_2) \vii \viii 
 + C \abs{\alpha} \, \eta^3 \abs{\ui}^2 
 + C \abs{\alpha} \, \eta \norm{\nabla u}_{L^2}^2 \,.
\end{align*}
For $I_4$ we obtain the estimate
\begin{align*}
 I_4 
 &= \eta^2 \int H[\ui] \cdot \ui \, dx+ 2 \eta^2 \int H[\ui] \cdot \uii \, dx+ \eta^2 \int H[\uii] \cdot \uii\, dx
 \le -\eta^2 \dt\ui \cdot \ui + C \eta^3 \abs{\ui}^2 + C \eta \norm{\nabla u}_{L^2}^2 \,.
\end{align*}
We rewrite $I_5$ as follows 
\begin{align*}
 I_5
 =& -\eta^2 \int (m_\eta - e_1) \cdot H[m_\eta - e_1] \abs{u}^2 \, dx- \eta^2 \int (m_\eta - e_1) \cdot
 H[e_1] \abs{u}^2\, dx
 - \eta^2 \int e_1 \cdot H[m_\eta - e_1] \abs{u}^2 \, dx
 \\
 &- \eta^2 \int e_1 \cdot H[e_1] \abs{u}^2\, dx
\end{align*}
and use the H{\"o}lder inequality to find
\begin{align*}
 I_5 \le& \eta^2 \norm{m_\eta - e_1}_{L^\infty} \norm{m_\eta - e_1}_{L^2} \norm{u}_{L^4}^2 
 + \eta^2 \norm{m_\eta - e_1}_{L^\infty} \norm{H[e_1]}_{L^2} \norm{u}_{L^4}^2
 + \eta^2 \norm{m_\eta - e_1}_{L^2} \norm{u}_{L^4}^2 
 \\
 &- \eta^2 \int e_1 \cdot H[e_1] \abs{u}^2 \, dx\,.
\end{align*}
We obtain with the help of the embedding $H^1(\Omega) \hookrightarrow L^4(\Omega)$
that
\begin{align*}
 I_5 \le& C \eta^3 ( \norm{u}_{L^2}^2 + \norm{\nabla u}_{L^2}^2) 
 - \eta^2 \int e_1 \cdot H[e_1] \abs{u}^2\, dx
 \le C \eta^3 \abs{\ui}^2 + C \eta^3 \norm{\nabla u}_{L^2}^2 
 - \eta^2 \int e_1 \cdot H[e_1] \abs{u}^2 \, dx\,.
\end{align*}
Moreover, we observe that
\begin{align*}
 - \eta^2 \int e_1 \cdot H[e_1] \abs{u}^2 \, dx
 =& \eta^2 \, \dt e_1 \cdot e_1 \, \abs{\ui}^2- 2 \eta^2 \int e_1 \cdot H[e_1] \ui \cdot \uii \, dx
 - \eta^2 \int e_1 \cdot H[e_1] \abs{\uii}^2\, dx \,.
\end{align*}
From here we get together with the H{\"o}lder inequality, Young inequality, and the Sobolev embedding $H^1(\Omega) \hookrightarrow L^4(\Omega)$ the estimate
\begin{align*}
 - \eta^2 \int e_1 \cdot H[e_1] \abs{u}^2 \, dx&\le \eta^2  \lambda_1 \abs{\ui}^2 + C \eta^2 
  \abs{\ui}\, \norm{\uii}_{L^2} + \eta^2 \norm{\uii}_{L^4}^2
\le \eta^2  \lambda_1 \abs{\ui}^2 + C \eta^3 \abs{\ui}^2 + C \eta \norm{\nabla u}_{L^2}^2 \,.
\end{align*}
We conclude
 $I_5 \le \eta^2 \lambda_1 \abs{\ui}^2 + C \eta^3 \abs{\ui}^2 + C \eta \norm{\nabla u}_{L^2}^2$.
Summarizing, we have shown so far that
\begin{align*}
 (-\mathcal{L}^\eta_0 u,u)_{L^2}  \ge& ( 1 - C \eta - C \abs{\alpha} \, \eta^{\frac{1}{2}} ) 
  \norm{\nabla u}_{L^2}^2
 + \eta^2 ( \dt\ui \cdot \ui - \lambda_1 \abs{\ui}^2 - \alpha (\lambda_3 - \lambda_2) \vii \viii -
 C \eta \abs{\ui}^2 
 \\
 &- C \abs{\alpha} \, \eta ^{\frac{1}{2}} \abs{\ui}^2 ) 
\end{align*}
for all $u \in T_{m_\eta}\mathcal{M}$. Moreover, we can write $\dt \ui \cdot \ui = \sum_{i=1}^3 \lambda_i \, \overline{u}_i^2$, hence
  \begin{align*}
    \dt \ui \cdot \ui - \lambda_1  \abs{\ui}^2 = \sum_{i=1}^3 \lambda_i \, \overline{u}_i^2 
    - \lambda_1 \sum_{i=1}^3 \overline{u}_i^2
     \ge (\lambda_2 - \lambda_1) \sum_{i=2}^3 \overline{u}_i^2 \,.
  \end{align*}
  Because of $u \cdot m_\eta = 0$ and \eqref{eq:assumptions for m}, we have that
  \begin{align*}
    \abs{\ui_1} = \left|\int u \cdot e_1 \, dx\right|= \left|\int u \cdot (e_1 - m_\eta) \, dx\right|\le \norm{u}_{L^2} \, 
    \norm{e_1-m_\eta}_{L^2} 
    \le C \eta \, \norm{u}_{L^2} \, ,
  \end{align*}
  thus
    $\ui_1^2 \le C \eta^2 \abs{\ui}^2 + C \eta^2 \norm{\nabla u}_{L^2}^2$.
  We obtain the estimate
  \begin{align*}
    \dt \ui \cdot \ui - \lambda_1 \, \abs{\ui}^2 
    &= (\lambda_2 -\lambda_1) \, \abs{\ui}^2 - (\lambda_2 - \lambda_1) \, \ui_1^2
    \ge  (\lambda_2 -\lambda_1) \, \abs{\ui}^2 - C  \eta^2 \, \abs{\ui}^2 
   - C  \eta^2 \norm{\nabla u}_{L^2}^2 \,.
  \end{align*}
 Finally, this shows
 \begin{align*}
 (-\mathcal{L}^\eta_0 u,u)_{L^2}  \ge& ( 1 - C \eta - C \abs{\alpha} \, \eta^{\frac{1}{2}} ) 
 \norm{\nabla u}_{L^2}^2
 \\
 &+ \eta^2 \big( (\lambda_2 - \lambda_1) \abs{\ui}^2 - \alpha (\lambda_3 - \lambda_2) \vii \viii -
 C \eta \abs{\ui}^2 - C \abs{\alpha} \, \eta ^{\frac{1}{2}} \abs{\ui}^2 \big)
 \end{align*}
 for all $u \in T_{m_\eta}\mathcal{M}$ and $\eta$ small enough. The lemma is proved.
\end{proof}
To state the second estimate, we introduce for a given $u \in T_{m_\eta}\mathcal{M}$ the test function $w[u] \in H^2_N(\Omega,\setR^3)$
defined by
  $w[u]= (0, -u_3,u_2)$.
It is easily seen that
$u \cdot w[u] = 0$ and $\int_\Omega \nabla u : \nabla w[u]\, dx = 0$
for every $u \in T_{m_\eta}\mathcal{M}$.
\begin{lem}
\label{lem:second estimate}
 There exist positive constants $\eta_0=\eta_0(\Omega)$ and $C=C(\Omega)$ such that
 \begin{align*}
  (-\mathcal{L}^\eta_0 u, \alpha w[u])_{L^2}  &\ge 
  \alpha \, \eta^2 (\lambda_3 - \lambda_2) \vii \viii
  -C (\alpha^2 + \abs{\alpha}) \eta^{\frac{1}{2}} \norm{\nabla u}_{L^2}^2
  -C (\alpha^2 + \abs{\alpha} ) \eta^{2+\frac{1}{2}} \abs{\ui}^2
 \end{align*}
 for every $u \in T_{m_\eta}\mathcal{M}$ and every $0<\eta\le\eta_0$.
\end{lem}
\begin{proof}
 Let $u \in T_{m_\eta}\mathcal{M}$ be given. For convenience we write $w=w[u]$, 
 $\int = \int_\Omega$ and find from \eqref{eq:identity for L0eta} the identity
 \begin{align*}
  &(-\mathcal{L}^\eta_0u,\alpha w)_{L^2}  
 \\
 =&- \alpha^2 \int (m_\eta \times \Delta u) \cdot w \, dx
  - 2 \alpha \int \nabla u :
 \nabla m_\eta \,\, m_\eta \cdot w \, dx
 - \alpha^2 \int (u \times \Delta m_\eta ) \cdot w \, dx
\\ 
&- \alpha^2 \, \eta^2 \int (m_\eta \times H[u]) \cdot w\, dx
 - \alpha^2 \, \eta^2 \int ( u \times H[m_\eta] ) \cdot w \, dx
  + \alpha \,\eta^2 \int m_\eta \cdot H[u] \, m_\eta \cdot w\, dx
 \\
 &- \alpha \,\eta^2 \int H[u] \cdot w \, dx+ \alpha \, \eta^2 \int u \cdot H[m_\eta] \, m_\eta \cdot w\, dx
 \\
 =& -I_1 - I_2 -I_3 -I_4 -I_5 - I_6 -I_7 -I_8 \,.
 \end{align*}
 We analyze each term separately. For $I_1$ we remark that $m_\eta \times w = (- m^\eta_1 u_1,- m^\eta_1 u_2,- m^\eta_1 u_3)$
 since $m_\eta \cdot u = m^\eta_1 u_1 + m^\eta_2 u_2 + m^\eta_3 u_3 = 0$. It follows
  \begin{align*}
   I_1 &= \alpha^2 \int (m_\eta \times \Delta u ) \cdot w\, dx
   = - \alpha^2 \int (m_\eta \times w) \cdot \Delta u\, dx
  = \alpha^2 \int (m^\eta_1 u_1 \Delta u_1 + m^\eta_1 u_2 \Delta u_2 + m^\eta_1 u_3 \Delta u_3 ) \, dx\, ,
  \end{align*}
  and by integration by parts, we obtain
  \begin{align*} 
   I_1 &= -\alpha^2 \int m^\eta_1 \abs{\nabla u}^2 \, dx
   - \alpha^2 \int ( u_1 \nabla u_1 +  u_2 \nabla u_2 +  u_3 \nabla u_3 ) \cdot \nabla m^\eta_1 \, dx\,.
  \end{align*}
  Thanks to $\norm{m_\eta - e_1}_{L^\infty} \le C_0 \eta$, we have $m^\eta_1 \ge 0$ for $\eta$ 
  small enough, hence
  \begin{align*}
   I_1 &\le 
   - \alpha^2 \int ( u_1 \nabla u_1 +  u_2 \nabla u_2 +  u_3 \nabla u_3 ) \cdot \nabla m^\eta_1 \, dx\,.
  \end{align*}
  Together with the H{\"o}lder inequality, Young inequality, \eqref{eq:estimate for u}, \eqref{eq:assumptions for m}, and the embedding $H^1(\Omega) \hookrightarrow L^4(\Omega)$, we get
  \begin{align*}
   I_1 
  & \le C \alpha^2 \, \eta^{2+\frac{1}{2}} \abs{\ui}^2 
  +  C \alpha^2 \, \eta^{\frac{1}{2}} \norm{\nabla u}_{L^2}^2 \,.
  \end{align*}
  Similarly, we find for $I_2$ the estimate $I_2\le C  \abs{\alpha} \, \eta^{2+\frac{1}{2}} \abs{\ui}^2 
   + C \abs{\alpha} \, \eta^{\frac{1}{2}}  \norm{\nabla u}_{L^2}^2$.
  For $I_3$ we obtain by integration by parts the identity
  \begin{align*}
   I_3 
  =& 2 \alpha^2 \int u_3 \nabla u_3 \cdot \nabla m^\eta_1 \, dx
  - \alpha^2 \int u_1 \nabla u_3 \cdot \nabla m^\eta_3 \, dx- \alpha^2 \int u_3 \nabla u_1 \cdot \nabla m^\eta_3\, dx
  \\
  & - \alpha^2 \int u_1 \nabla u_2 \cdot \nabla m^\eta_2 \, dx
  - \alpha^2 \int u_2 \nabla u_1 \cdot \nabla m^\eta_2 \, dx
  + 2 \alpha^2 \int u_2 \nabla u_2 \cdot \nabla m^\eta_1 \, dx\,.
  \end{align*}
  With the usual inequalities, we get $I_3\le C \alpha^2 \eta^{2+\frac{1}{2}} \abs{\ui}^2
   + C \alpha^2 \eta^{\frac{1}{2}} \norm{\nabla u}_{L^2}^2$.
  We decompose $I_4$ as follows:
  \begin{align*}
   I_4 &= \alpha^2 \eta^2 \int \big( (m_\eta-e_1) \times H[u] \big) \cdot w \, dx
   + \alpha^2 \eta^2 \int (e_1 \times H[u] ) \cdot w\, dx
  = I_4^1 + I_4^2 \,.
  \end{align*}
  For $I_4^1$ we find with the H{\"o}lder inequality, \eqref{eq:estimate for u}, and \eqref{eq:assumptions for m} the estimate
  \begin{align*}
   I_4^1 \le \alpha^2 \eta^2 \norm{m_\eta - e_1}_{L^\infty} \norm{u}_{L^2}^2 
   \le C \alpha^2 \eta^3 \abs{\ui}^2 +  C \alpha^2 \eta^3 \norm{\nabla u}_{L^2}^2 \,.
  \end{align*}
  We decompose $I_4^2$ with the help of \eqref{eq:decomposition} as follows:
  \begin{align*}
   I_4^2
  =& \alpha^2 \eta^2 \int (e_1 \times H[\ui]) \cdot \uwi\, dx
  + \alpha^2 \eta^2 \int (e_1 \times H[\ui]) \cdot \uwii\, dx
  \\
  &+ \alpha^2 \eta^2 \int (e_1 \times H[\uii]) \cdot \uwi\, dx
  + \alpha^2 \eta^2 \int (e_1 \times H[\uii]) \cdot \uwii \, dx\,.
  \end{align*}
  We obtain with the H{\"o}lder inequality, Young inequality, and \eqref{eq:estimate for uii} the estimate
  \begin{align*}
  I_4^2
  \le & \alpha^2 \eta^2 \int (e_1 \times H[\ui]) \cdot \uwi \, dx
  + C \alpha^2 \eta^3 \abs{\ui}^2 + C \alpha^2 \eta \norm{\nabla u}_{L^2}^2 \,.
  \end{align*}
  Moreover, we observe with the help of the demagnetizing tensor $\dt$ that
 \begin{align*}
  \alpha^2 \eta^2 \int (e_1 \times H[\ui]) \cdot \uwi \, dx &= 
  - \alpha^2 \eta^2 ( e_1 \times \dt\ui)\cdot \uwi 
  = -\alpha^2 \eta^2 ( \lambda_3 \viii^2 + \lambda_2 \vii^2 ) 
 \end{align*}
  and conclude
  \begin{align*}
   I_4 \le -\alpha^2 \eta^2 ( \lambda_3 \viii^2 + \lambda_2 \vii^2 )
   + C \alpha^2 \eta^3 \abs{\ui}^2 + C \alpha^2 \eta \norm{\nabla u}_{L^2}^2 \,.
  \end{align*}
  We decompose $I_5$ as follows:
  \begin{align*}
   I_5 &= \alpha^2 \eta^2 \int ( u \times H[m_\eta - e_1] ) \cdot w \, dx
   + \alpha^2 \eta^2 \int ( u \times H[e_1] ) \cdot w\, dx
  = I_5^1 + I_5^2 \,.
  \end{align*}
  With the H{\"o}lder inequality, the embedding $H^1(\Omega) \hookrightarrow L^4(\Omega)$, \eqref{eq:estimate for u}, and \eqref{eq:assumptions for m}, we find the estimate
  \begin{align*}
   I_5^1 &\le \alpha^2 \eta^2 \norm{u}_{L^4}^2 \norm{m_\eta - e_1}_{L^2} 
   \le C \alpha^2 \eta^3 \abs{\ui}^2 + C \alpha^2 \eta^3 \norm{\nabla u}_{L^2}^2 \,.
  \end{align*}
  In view of \eqref{eq:decomposition}, we split the second term as follows:
 \begin{align*}
  I_5^2 
  =& \alpha^2 \eta^2 \int (\ui \times H[e_1]) \cdot \uwi\, dx
  + \alpha^2 \eta^2 \int (\ui \times H[e_1]) \cdot \uwii\, dx
 \\
 &+ \alpha^2 \eta^2 \int (\uii \times H[e_1]) \cdot \uwi\, dx
  +\alpha^2 \eta^2 \int (\uii \times H[e_1]) \cdot \uwii \, dx\,.
 \end{align*}
 With the usual arguments, we obtain
 \begin{align*}
  I_5^2 \le & \alpha^2 \eta^2 \int (\ui \times H[e_1]) \cdot \uwi \, dx+  C \alpha^2 \eta^3 \abs{\ui}^2 
  + C \alpha^2 \eta \norm{\nabla u}_{L^2}^2 \,.
 \end{align*}
Moreover, we have thanks to the demagnetizing tensor $\dt$ that
\begin{align*}
 \alpha^2 \eta^2 \int (\ui \times H[e_1]) \cdot \uwi \, dx=& 
 - \alpha^2 \eta^2 (\ui \times \dt e_1 ) \cdot \uwi
 = \alpha^2 \eta^2 ( \lambda_1 \viii^2 + \lambda_1 \vii^2) \, ,
\end{align*}
and this shows
\begin{align*}
 I_5 \le \alpha^2 \eta^2 ( \lambda_1 \viii^2 + \lambda_1 \vii^2)
 +  C \alpha^2 \eta^3 \abs{\ui}^2 
  + C \alpha^2 \eta \norm{\nabla u}_{L^2}^2 \,.
\end{align*}
Since $w \cdot e_1 = 0$, we find for $I_6$ together with \eqref{eq:estimate for u} and \eqref{eq:assumptions for m} that
\begin{align*}
 I_6 &= - \alpha \, \eta^2 \int m_\eta \cdot H[u] \, (m_\eta - e_1)\cdot w \, dx
 \le C \abs{\alpha} \, \eta^3 \abs{\ui}^2 +  C \abs{\alpha} \, \eta^3 \norm{\nabla u}_{L^2}^2 \,.
\end{align*}
We decompose $I_7$ as follows
\begin{align*}
 I_7 
 =& \alpha \, \eta^2 \int H[\ui] \cdot  \uwi \, dx+ \alpha \, \eta^2 \int H[\ui] \cdot  \uwii\, dx
 + \alpha \, \eta^2 \int H[\uii] \cdot  \uwi\, dx
  + \alpha \, \eta^2 \int H[\uii] \cdot  \uwii\, dx
\end{align*}
and find as usual the estimate
\begin{align*}
 I_7 &\le \alpha \, \eta^2 \int H[\ui] \cdot  \uwi \, dx+ C \abs{\alpha} \, \eta^3 \abs{\ui}^2 +
 C \abs{\alpha} \, \eta \norm{\nabla u}_{L^2}^2 \,.
\end{align*}
Furthermore, we see with the help of the demagnetizing tensor $\dt$ that
\begin{align*}
 \alpha \, \eta^2 \int H[\ui] \cdot \uwi\, dx &= - \alpha \, \eta^2 \dt\ui \cdot \uwi 
= - \alpha \, \eta^2 (\lambda_3 - \lambda_2) \vii \viii 
\end{align*}
and conclude
\begin{align*}
 I_7 \le - \alpha \, \eta^2 (\lambda_3 - \lambda_2) \vii \viii
 + C \abs{\alpha} \, \eta^3 \abs{\ui}^2 +
 C \abs{\alpha} \, \eta \norm{\nabla u}_{L^2}^2 \,.
\end{align*}
We make again use of $w \cdot e_1 = 0$ and obtain
\begin{align*}
 I_8 &= - \alpha \, \eta^2 \int u \cdot  H[m_\eta] \, (m_\eta - e_1 ) \cdot w \, dx
 \le C \abs{\alpha} \, \eta^2 \norm{m_\eta - e_1}_{L^\infty} \norm{u}_{L^4}^2 
 \le C \abs{\alpha} \, \eta^3 \abs{\ui}^2 +  C \abs{\alpha} \, \eta^3 \norm{\nabla u}_{L^2}^2 \,.
\end{align*}
Summarizing, we have the estimate
\begin{align*}
 (-\mathcal{L}^\eta_0 u , \alpha w)_{L^2}  \ge& 
  \alpha \, \eta^2 (\lambda_3 - \lambda_2) \vii \viii
 + \alpha^2 \eta^2 ( \lambda_3 - \lambda_1) \viii^2 + \alpha^2 \eta^2 (\lambda_2 
 - \lambda_1 ) \vii^2
 \\
 &-C (\alpha^2 + \abs{\alpha}) \eta^{\frac{1}{2}} \norm{\nabla u}_{L^2}^2
 -C (\alpha^2 + \abs{\alpha} ) \eta^{2+\frac{1}{2}} \abs{\ui}^2 \,.
\end{align*}
Since $\lambda_3,\lambda_2 > \lambda_1$, the statement of the lemma follows.
\end{proof}
A combination of Lemmas \ref{lem:first estimate} and \ref{lem:second estimate} yields the desired result:
\begin{lem}
 \label{lem:second step}
 There is a positive constant $\eta_0 = \eta_0 (\Omega,\alpha)$ such that 
 $0$ does not belong to the point spectrum  of $\mathcal{L}^\eta_0$ for every $0<\eta\le\eta_0$.
\end{lem}

\section{Existence of $T$-periodic solutions}
\label{sec:main result - small particles}
Finally, we can state our main result, which is to the author's best knowledge the first existence result concerning time-periodic solutions for the three-dimensional LLG.
\begin{theorem*}
 Suppose that $\Omega \subset \setR^3$ is a bounded $C^{2,1}$-domain with $\abs{\Omega}=1$ such that the smallest eigenvalue of the 
 demagnetizing tensor $\dt$ is simple. Furthermore let $h \in C^{0,\beta}(\setR,L^2(\Omega,\setR^3))$ be $T$-periodic, where $T>0$ and $0<\beta<1$.
 Then there is a positive constant $\eta_0 = \eta_0 (\Omega,\alpha)$ 
 with the following property: For every $0<\eta\le\eta_0$ there exists an open neighborhood
 $V=V(\eta,\alpha,h,\Omega)$ of $0$ in $\setR$ such that the rescaled LLG
  possesses a $T$-periodic solution
   $m \in C^1\big(\setR,L^2(\Omega,\setR^3)\big) \cap C\big(\setR,H^2_N(\Omega,\setR^3)\big)$
  for every $\lambda \in V$. By scaling the statement carries over to the original LLG on $\Omega_\eta = \eta \, \Omega$.
\end{theorem*}
\begin{proof}
Since $\Omega$ is a bounded $C^{2,1}$-domain, we know that minimizers $m_\eta$ of $E^\eta_\text{res}$ are regular for every $0<\eta\le\eta_0(\Omega)$. Let now $0<\eta\le\eta_0(\Omega)$ be given and $m_\eta$ be a minimizer of $E^\eta_\text{res}$. The regularity assumptions on $h$ imply together with Lemma \ref{lem:existence} the existence of a unique solution $m(\cdot,u,\lambda)$ for $(LLG)_\eta$ close to $m_\eta$ for all $u \in U$, $\lambda \in V$, where $U$ is an open neighborhood of $m_\eta$ in $H^2_N(\Omega,\setR^3)$, and $V$ is an open neighborhood of $0$ in $\setR$. Moreover, the mapping
$ f:U \times V \to H^2_N(\Omega,\setR^3): (u,\lambda) \mapsto m(T,u,\lambda)$
is smooth, and because of Lemma \ref{lem:punctual norm}, we have $f(U\cap \mathcal{M} \times V) \subset \mathcal{M}$. Since $h$ is $T$-periodic, we know that $m(\cdot,u,\lambda)$ defines a $T$-periodic solution for $(LLG)_\eta$ with saturation constraint $\abs{m}=1$ if and only if $f(u,\lambda) = u$ and $u \in \mathcal{M}$. In particular, it is enough to solve the parameter dependent fixed point equation $f(u,\lambda)=u$
on $\mathcal{M}$. With the help of Lemma \ref{lem:existence} we find that $D_1f(m_\eta,0)u = e^{T \mathcal{L}^\eta_0}u$ for $u \in T_{m_\eta}\mathcal{M}$.
Thanks to Lemmas \ref{lem:first step} and \ref{lem:second step} we know that $\sigma_P(\mathcal{L}^\eta_0) \cap \dot{\imath} \setR = \emptyset$ whenever $0<\eta\le\eta_0(\Omega,\alpha)$. Now, the discussion at the end of Section \ref{sec:continuation} shows that $D_1f(m_\eta,0) -I$ is invertible on $T_{m_\eta}\mathcal{M}$, and the theorem follows from Lemma \ref{lem:continuation}.
\end{proof}

\medskip

\noindent {\bf Acknowledgments.}
This work is part of the author's PhD thesis prepared at the Max Planck Institute for Mathematics in the Sciences (MPIMiS) and submitted in June 2009 at the University of Leipzig, Germany. The author would like to thank his supervisor Stefan M{\"u}ller for the opportunity to work at MPIMiS and for having chosen an interesting problem to work on. Financial support from the International Max 
Planck Research School `Mathematics in the Sciences' (IMPRS) is also acknowledged.

\bibliographystyle{abbrv}
\bibliography{article-small}

\begin{thebibliography}{10}

\bibitem{Abraham}
R.~Abraham, J.~E. Marsden, and T.~Ratiu.
\newblock {\em Manifolds, tensor analysis, and applications}, volume~75 of {\em
  Applied Mathematical Sciences}.
\newblock Springer-Verlag, New York, second edition, 1988.

\bibitem{aharoni2}
A.~Aharoni.
\newblock {\em Introduction to the theory of ferromagnetism}.
\newblock Oxford University Press, 1996.

\bibitem{Amann2}
H.~Amann.
\newblock Dynamic theory of quasilinear parabolic equations. {II}.
  {R}eaction-diffusion systems.
\newblock {\em Differential Integral Equations}, 3(1):13--75, 1990.

\bibitem{Amann}
H.~Amann.
\newblock {\em Ordinary differential equations}, volume~13 of {\em de Gruyter
  Studies in Mathematics}.
\newblock Walter de Gruyter \& Co., Berlin, 1990.
\newblock An introduction to nonlinear analysis, Translated from the German by
  Gerhard Metzen.

\bibitem{brown}
W.~Brown.
\newblock {\em Micromagnetics}.
\newblock Wiley, 1963.

\bibitem{carbou}
G.~Carbou.
\newblock Regularity for critical points of a nonlocal energy.
\newblock {\em Calc. Var. Partial Differential Equations}, 5(5):409--433, 1997.

\bibitem{carboufabrie}
G.~Carbou and P.~Fabrie.
\newblock Regular solutions for {L}andau-{L}ifschitz equation in a bounded
  domain.
\newblock {\em Differential Integral Equations}, 14(2):213--229, 2001.

\bibitem{DeSimone}
A.~De~Simone.
\newblock Hysteresis and imperfection sensitivity in small ferromagnetic
  particles.
\newblock {\em Meccanica}, 30(5):591--603, 1995.
\newblock Microstructure and phase transitions in solids (Udine, 1994).

\bibitem{dkmo}
A.~DeSimone, R.~V. Kohn, S.~M{\"u}ller, and F.~Otto.
\newblock Recent analytical developments in micromagnetics.
\newblock {\em in: Science of Hysteresis, Elsevier, G. Bertotti and I
  Magyergyoz, Eds.}, 2005.

\bibitem{GuoDing}
B.~Guo and S.~Ding.
\newblock {\em Landau-{L}ifshitz equations}, volume~1 of {\em Frontiers of
  Research with the Chinese Academy of Sciences}.
\newblock World Scientific Publishing Co. Pte. Ltd., Hackensack, NJ, 2008.

\bibitem{hardtkinderlehrer}
R.~Hardt and D.~Kinderlehrer.
\newblock Some regularity results in ferromagnetism.
\newblock {\em Comm. Partial Differential Equations}, 25(7-8):1235--1258, 2000.

\bibitem{alex_diss}
A.~Huber.
\newblock {\em Periodic solutions for the {L}andau-{L}ifshitz-{G}ilbert
  equation}.
\newblock PhD thesis, Universit\"at Leipig, 2009.

\bibitem{alex_regularity}
A.~Huber.
\newblock Boundary regularity for minimizers of the micromagnetic energy
  functional.
\newblock {\em arXiv}, 2010.

\bibitem{alex_walls}
A.~Huber.
\newblock Time-periodic {N}{\'e}el wall motions.
\newblock {\em arXiv}, 2010.

\bibitem{HS}
A.~Hubert and R.~Sch{\"a}fer.
\newblock {\em Magnetic Domains}.
\newblock Springer-Verlag, 1998.

\bibitem{Kato}
T.~Kato.
\newblock {\em Perturbation theory for linear operators}.
\newblock Classics in Mathematics. Springer-Verlag, Berlin, 1995.
\newblock Reprint of the 1980 edition.

\bibitem{lunardi}
A.~Lunardi.
\newblock {\em Analytic semigroups and optimal regularity in parabolic
  problems}.
\newblock Progress in Nonlinear Differential Equations and their Applications,
  16. Birkh\"auser Verlag, Basel, 1995.

\bibitem{Stoner}
E.~Stoner and E.~Wohlfarth.
\newblock A mechanism of magnetic hysteresis in heterogeneous alloys.
\newblock {\em Phil. Trans. Roy. Soc. London}, A240:599--642, 1948.

\end{thebibliography}

\bigskip\small

\noindent{\sc NWF I-Mathematik, Universit\"at Regensburg,  93040 Regensburg}\\
{\it E-mail address}: {\tt alexander2.huber@mathematik.uni-regensburg.de}

\end{document}